\documentclass[a4paper,12pt, reqno]{amsart}
\usepackage{color}
\usepackage{a4wide}
\usepackage{geometry}
\usepackage{hyperref}
\usepackage{amsmath}
\usepackage{amssymb}
\usepackage{amsthm}
\usepackage[active]{srcltx}
\usepackage[dvips]{graphicx}

\numberwithin{equation}{section}
\newtheorem{theorem}{Theorem}[section]

\def\max{{\rm max}}

\def\det{{\rm det}}

\def\Int{{\rm Int}}

\def\Sym1{\mathcal{S}(1)}
\def\cA{{\mathcal A}}

\def\cG{{\mathcal G}}

\def\cJ{{\mathcal J}}

\def\cM{{\mathcal M}}

\def\cP{{\mathcal P}}

\def\cS{{\mathcal S}}
\newcommand{\veps}{\varepsilon}

\newcommand{\F}{\mathcal{F}}
\newcommand{\G}{\mathcal{G}}

\newcommand{\wt}{\widetilde}
\newcommand{\FD}{\wt{\F}}
\newcommand{\osc}{\operatornamewithlimits{{\rm osc}}}

\newcommand{\note}[1]{\marginpar{\tiny\emph{#1}}}

\newtheorem{thm}{\textbf{Theorem}}[section]
\newtheorem{lem}[thm]{\textbf{Lemma}}

\theoremstyle{remark}
\newtheorem{rem}[thm]{\textbf{Remark}}

\newtheorem{cor}[thm]{\textbf{Corollary}}

\newtheorem{exe}[thm]{\textbf{Example}}

\theoremstyle{definition}
\newtheorem{defn}[thm]{{Definition}}

\newtheoremstyle{Claim}{}{}{\itshape}{}{\itshape\bfseries}{:}{ }{#1}
\theoremstyle{Claim}

\newcommand{\N}{{\mathbb N}}

\newcommand{\R}{{\mathbb R}}
\newcommand{\Cn}{{\mathbb C}}
\newcommand{\Hn}{{\mathbb H}}

\DeclareMathOperator{\Diff}{Diff}

\newcommand{\USC}{\mathrm{USC}}
\newcommand{\LSC}{\mathrm{LSC}}

\makeatletter
\newcommand{\tpitchfork}{%
	\vbox{
		\baselineskip\z@skip
		\lineskip-.52ex
		\lineskiplimit\maxdimen
		\m@th
		\ialign{##\crcr\hidewidth\smash{$-$}\hidewidth\crcr$\pitchfork$\crcr}
	}%
}
\makeatother

\begin{document}

\title[Alexandrov estimates by determinant majorization]{Alexandrov estimates for polynomial operators by determinant majorization}
\author[F.R. Harvey]{F. Reese Harvey}
\address{Department of Mathematics\\ Rice University\\ P.O. Box 1892\\ Houston, TX 77005-1892, USA}
\email{harvey@rice.edu (F. Reese Harvey)}
%\author[H.B. Lawson]{H. Blaine Lawson}
%\address{Department of Mathematics\\ Stony Brook University\\ Stony Brook, NY 11794-3651, USA}
%\email{blaine@math.stonybrook.edu (H.\ Blaine Lawson, Jr.)}
\author[K.R. Payne]{Kevin R. Payne}
\address{Dipartimento di Matematica ``F. Enriques''\\ Universit\`a di Milano\\ Via C. Saldini 50\\ 20133--Milano, Italy}
\email{kevin.payne@unimi.it (Kevin R.\ Payne)}\thanks{Payne partially supported by the Gruppo Nazionale per l'Analisi Matematica, la Probabilit\`a e le loro Applicazioni (GNAMPA) of the Istituto Nazionale di Alta Matematica (INdAM) and the project: GNAMPA 2024 ``Free boundary problems in noncommutative
	structures and degenerate operators''.}

\date{\today} \linespread{1.2}



\maketitle

\makeatletter
\def\l@subsection{\@tocline{2}{0pt}{2.5pc}{5pc}{}}
\makeatother

	\begin{abstract}
		We obtain estimates on the supremum, infimum and oscillation of solutions for a wide class of inhomogeneous fully nonlinear elliptic equations on Euclidean domains where the differential operator is an {\em $I$-central G\aa rding-Dirichlet operator} in the sense of Harvey-Lawson (2024). The argument combines two recent results: an Alexandrov estimate of Payne-Redaelli (2025) for locally semiconvex functions based on the area formula and a determinant majorization estimate of Harvey-Lawson (2024).
		% extending the one in \cite{HL23}. 
		The determinant majorization estimate has as a special case the arithmetic - geometric mean inequality, so the result includes the classical Alexandrov-Bakelman-Pucci estimate for linear operators. A potential theoretic approach is used involving subequation subharmonics and their dual subharmonics. Semiconvex approximation plays a crucial role.
	\end{abstract}

	\maketitle

\setcounter{tocdepth}{1}
\tableofcontents

%\changelocaltocdepth{2}

%%non so a che serve
\makeatletter
\def\l@subsection{\@tocline{2}{0pt}{2.5pc}{5pc}{}}
\makeatother

\setcounter{tocdepth}{1}
\tableofcontents
	
	\section{Introduction}\label{sec:intro}
	
	This paper is concerned with estimates for viscosity solutions $h$  on bounded open sets $\Omega \subset \R^n$ of inhomogeneous degenerate elliptic fully nonlinear equations of the form
	\begin{equation}\label{GDE}
	\mathfrak{g}(D_x^2h) = f(x), \quad x \in \Omega,
	\end{equation}
	where $f \in C(\Omega)$ and $\mathfrak{g}$ is a homogeneous polynomial operator acting on the Hessian of $h$ which takes values in $\cS(n)$, the space of real symmetric $n \times n$ matrices.  We assume that $\mathfrak{g}$ is an {\bf{\em{$I$-central G\aa rding-Dirichlet polynomial}}} (see Definition \ref{defn:ICGD_poly}). 
	%Appendix \ref{sec:App_A} provides a discussion \note{We might want to add to Appendix A things like admissibility constraints and the definitions of sub and supersolutions, which I did for all jet variables (see the discussion after formula (1.2) below}  
	See Definition \ref{defn:DGsolns} for the precise notions of viscosity solution.
	
	The estimates are pointwise {\em a priori} estimates which bound the supremum, infimum and oscillation on $\Omega$ in terms the same quantities on the boundary $\partial \Omega$ plus an additive ``error term'' which, in the case of dual subsolutions to \eqref{GDE}, quantifies the failure of the maximum principle and is often called an Alexandrov maximum principle. Such estimates play a fundamental role in the study of fully nonlinear elliptic PDE and are also often called ABP estimates (for Alexandrov, Bakelman and Pucci). For example, see Chapter 5 of \cite{CC}. 
	
	Our oscillation estimate for viscosity solutions $h \in C(\overline{\Omega})$ to \eqref{GDE} is
	\begin{equation}\label{AOE1}
	\osc_{\Omega} h \leq \osc_{\partial \Omega} h + \frac{\mathrm{diam}(\Omega)}{|B_1|^{1/n}\mathfrak{g}(I)^{1/N}} || f ^{1/N}||_{L^n(\Omega)},
	\end{equation}
	where  $\osc_{\Omega} h := \sup_{\Omega} h - \inf_{\Omega} h$, and similarly for the boundary oscillations (see Theorem \ref{thm:ABP_P1}). Also, $N$ is the degree of $\mathfrak{g}$,  ${\rm diam}(\Omega)$ is the diameter of $\Omega$ and $|B_1|$ is the Lebesgue measure of the unit ball $B_1$ in $\R^n$.

	On the other hand, for more regular solutions $h \in C^{1,1}(\Omega) \cap C(\overline{\Omega})$, one has an improved estimate
	with the integral over the subset $E^{-}(h) = E^{+}(-h)$ of $\Omega$
	\begin{equation}\label{AOE2}
	\osc_{\Omega} h \leq \osc_{\partial \Omega} h + \frac{\mathrm{diam}(\Omega)}{|B_1|^{1/n}\mathfrak{g}(I)^{1/N}} || f ^{1/N}||_{L^n(E^{-}(h))}.
	\end{equation}
	See Remark \ref{rem:ABP_C11} for the proof of \eqref{AOE2}. 
	%\note{Reese asks if we have an example of $h \in C(\overline{\Omega})$ where \eqref{AOE2} fails} 
	The subset in \eqref{AOE2} is the intersection
	\begin{equation}\label{AOE3}
	E^{-}(h) := \Gamma^{-}(h) \cap \Diff^2(h) = 	E^{+}(-h) := \Gamma^{+}(-h) \cap \Diff^2(-h) \subset \Omega.
	\end{equation}
	$\Diff^2(h) $ denotes the set of points $x \in \Omega$ at which $h$ is twice differentiable (in the Peano sense) and  $\Gamma^{\pm}(h)$ are the set of flat upper/lower contact points as defined in Definition \ref{defn:UC_LC} (points $x \in \Omega$ that admit a hyperplane that lies globally above/below the graph of $h$ at $x$). $\Diff^2(h)$ is of full Lebesgue measure in $\Omega$. This is because $h \in C^{1,1}(\Omega)$ is equivalent to $h$ being both locally semiconvex and locally semiconcave (see, for example,  Proposition 2.5 of \cite{PR25}). The second order differentiability claim follows from an easy extension of Alexandrov's theorem to locally semiconvex functions (see Theorem 2.2 of \cite{PR25}). The integral term in \eqref{AOE2} quantifies the presence of convexity in $h$ that controls how much $h$ can decrease from its boundary values. In theory, the integral estimate might involve the upper contact set which controlls how much $h$ might increase from its boundary values; but, as we will see, subsolutions are lapalacian subharmonics and hence the maximum principle holds (see Lemma \ref{lem:Gamma_MP}). The less precise estimate \eqref{AOE1} for $h$ merely continuous (in which the Lebesgue norm is taken over all of $\Omega$) is related to the method of proof, which makes use of semicovex approximations by way of the sup-convolution for $-h$. Using all of $\Omega$ is useful in practice since a priori one does not know much about the size of $\Gamma^{-}(h) = \Gamma^{+}(-h)$.
	
	%We have not tracked the upper/lower contact sets along the approximation; the same ``lazyness'' is also typically done in the passage from the classical Alexandrov maximum principle from $C^2(\overline{\Omega})$ to $W^{2,n}(\Omega) \cap C(\overline{\Omega})$ by way of mollification (integral convolution). Moreover, using all of $\Omega$  is not troublesome since a priori one does not know $\Gamma^{\pm}(h)$.
	
	Before going further, we will make precise the main class of operators we treat. 
	%At the end of the introduction, we will \note{Check this!} outline some possible extensions.
	%A first class of differential operators we consider are obtained by applying main classthe following class of polynomials to the Hessian of functions.
	
	\begin{defn}\label{defn:ICGD_poly}
		A real homogeneous polynomial $\mathfrak{g}$ of degree $N$ on $\cS(n)$ is said to be an {\em $I$-central G\aa rding-Dirichlet polynomial} if it satisfies the following three conditions.
		\begin{itemize}
			\item[(1)] $\mathfrak{g}$ is {\em $I$-hyperbolic}: for each $A \in \cS(n)$ fixed, the polynomial $\varphi_A(t) := \mathfrak{g}(tI + A)$ has only real roots $t = t_1^{\mathfrak{g}}(A), \ldots t_N^{\mathfrak{g}}(A)$. The negatives of these roots $\lambda_k^{\mathfrak{g}}(A) := - t_k^{\mathfrak{g}}(A)$ are called the {\em G\aa rding $I$-eigenvalues of $A$}. (We can and do normalize so that $\mathfrak{g}(I) > 0$).
			\item[(2)] $\mathfrak{g}$ is {\em G\aa rding-Dirichlet}: $\mathfrak{g}$ is $I$-hyperbolic and the (open) {\em G\aa rding cone} $\Gamma$, defined as the connected component containing $I$ of the set $\{ A \in \cS(n): \ \mathfrak{g}(A) \neq 0 \}$, contains
			\begin{equation}\label{GD_cone}
			\Int \, \cP := \{ A \in \cS(n): \ A > 0 \}.
			\end{equation}
			\item[(3)] $\mathfrak{g}$ is {\em $I$-central}: there exists $k > 0$ such that the gradient of $\mathfrak{g}$ at $I$ satisfies
			\begin{equation}\label{I_central}
			D_{I} \mathfrak{g} = kI.
			\end{equation}
		\end{itemize} 
	\end{defn}
	There is much to say about polynomials $\mathfrak{g}$ which satisfy one, or more of the conditions (1), (2) and (3). Here are a few of the major points.
	
	Condition (1) means that $\mathfrak{g}$ is a G\aa rding polynomial on the vector space $V = \cS(n)$. In terms of the G\aa rding $I$-eigenvalues of $A$ one has
	\begin{equation}\label{GMA_op}
	\mathfrak{g}(tI + A) = \prod_{k=1}^N \left( t + \lambda_k^{\mathfrak{g}}(A) \right) \quad \text{and hence} \quad \mathfrak{g}(A) = \prod_{k=1}^N \lambda_k^{\mathfrak{g}}(A),
	\end{equation}
	so that these polynomial operators can be considered as {\em generalized Monge-Amp\`ere operators} $\mathfrak{g}(A)$ generalizing  ${\rm det}(A)$ where the  G\aa rding $I$-eigenvalues of $A$ are just the usual eigenvalues of $A \in \cS(n)$. G\aa rding polynomials were introduced and their basic properties were laid out in G\aa rding's seminal paper\cite{Ga59} for the study of linear hyperbolic equations. For a discussion within our current context of nonlinear elliptic operators see \cite{HL13a}, \cite{HL10} and references therein.
	
	Since $\overline{\Gamma}$ is a convex cone, the condition (2) that $\cP \subset \overline{\Gamma}$ is equivalent to $\Gamma$ being $\cP$-monotone; that is, $\Gamma + \cP \subset \overline{\Gamma}$. In this form, condition (2) was introduced in \cite{HL09} as the key property of a subequation to be able to discuss viscosity solutions of \eqref{GDE}. It ensures that the induced differential operator $\mathfrak{g}(D^2u)$ is degenerate elliptic when restricted to functions $u$ whose Hessian $D^2_xu$ (in the viscosity sense) belongs to $\overline{\Gamma}$. See Appendix \ref{sec:App_A} for more details. A rich potential theoretic approach to existence and uniqueness for solutions $h$ to the Dirichlet problem (for domains which are strictly convex) associated to \eqref{GDE} has been developed (see \cite{HL09}, \cite{HL11}, \cite{CP17}, \cite{HL18c}, \cite{CHLP20}, \cite{CP21}, \cite{HP23} and references therein).
	
	The $I$-central condition (3) was introduced in \cite{HL24} as a key sufficient condition for {\bf{\em{determinant majorization}}} (see Theorem \ref{thm:DME})
	\begin{equation}\label{DME_intro}
	\mathfrak{g}(A)^{1/N} \geq 	\mathfrak{g}(I)^{1/N} \left( {\rm det} A \right)^{1/n}, \quad \forall \, A \geq 0 \ \text{in} \ \cS(n).
	\end{equation}
	This majorization estimate has had a profound impact for obtaining uniform pointwise estimates, local regularity and uniform entropy and energy bounds for classes of natural geometric operators on complex manifolds. It was pioneered in the groundbreaking work of Guo, Phong and Tong \cite{GPT23} with many recent developments in Guo-Phong \cite{GP24a} and \cite{GP24b}, amongst others.
	
	In this paper, the importance of the determinant majorization estimate \eqref{DME_intro} is that it allows one to replace $|{\rm det} D^2w|$ by $\mathfrak{g}(D^2w)$ in the Alexandrov estimate of Theorem \ref{thm:ABP_lqc}
	\begin{equation} \label{ABP_lqc_intro}
	\sup_{\Omega} w \leq \sup_{\partial \Omega} w + \frac{\mathrm{diam}(\Omega)}{|B_1|^{1/n}} \left( \int_{	E^{+}(w)} |\det \,  D^2w(x)| \, dx \right)^{1/n}
	\end{equation}
	when $w$ is locally semiconvex and satisfies $D^2w(x) \geq 0$ almost everywhere on $E^+(w)$. This substitution will lead to the oscillation estimate \eqref{AOE1} for viscosity solutions of \eqref{GDE} if $\mathfrak{g}$ is an $I$-central G\aa rding-Dirichlet polynomial operator. We briefly explain now why this is so.
	
	We will see, in Lemma \ref{lem:Gamma_MP}, that all {\em admissible subsolutions} $u \in \USC(\overline{\Omega})$ of \eqref{GDE} in the sense of Definition \ref{defn:DGsolns}(a) are classical laplacian subharmonics and hence satisfy the maximum principle
	\begin{equation}\label{MP_intro}
	\sup_{\Omega} u \leq \sup_{\partial \Omega} u.
	\end{equation}
	This is an improvement of the supremum estimate \eqref{ABP_lqc_intro} with no additive integral term involving the forcing term $f$. Consequently, it reduces the oscillation estimate to proving an estimate on the infimum for {\em admissible supersolutions} $v \in \LSC(\overline{\Omega})$ of \eqref{GDE} in the sense of Definition \ref{defn:DGsolns}(b). This infimum estimate is equivalent to the supremum estimate 
	\begin{equation}\label{ABPDG_dual_intro}
	\sup_{\Omega} w \leq \sup_{\partial \Omega} w + \frac{\mathrm{diam}(\Omega)}{|B_1|^{1/n}\mathfrak{g}(I)^{1/N}} || f ^{1/N}||_{L^n(\Omega)}.
	\end{equation}
	where $w:=-v \in \USC(\overline{\Omega})$ is a {\em dual subharmonic} (see the remarks following Definition \ref{defn:DGsolns}). This means that for each $x \in \Omega$ , $w$ satisfies
	\begin{equation}\label{F_dual}
	J^{2,+}_x w \in \wt{\F}_x := \{(r,p,A) \in  \cJ^2: \mbox{$-A \not\in \overline{\Gamma}$ or  [$-A \in \overline{\Gamma}$ and $\mathfrak{g}(-A) - f(x) \leq 0$]}\}, 
	\end{equation}
	where $J^{2,+}_x w$ is the space of all upper contact jets of $w$ at $x$ (see Definition \ref{defn:UCJ}). When $w$ is also locally semiconvex, one has $A = D^2 w(x) \leq 0$ on $E^+(w)$ and hence $-A \in \cP \subset \overline{\Gamma}$ so that $\mathfrak{g}(-D^2 w(x)) \leq f(x)$ and one completes the estimate if $w$ is locally semiconvex. Finally, the procedure of seminconvex approximation utilizing the fiberegularity (as developed in \cite{CPR23}) is used to eliminate the locally semiconvex hypothesis (see Lemma \ref{lem:QCA}).
	
	The paper is organized as follows. Section \ref{sec:ABP-lqc} contains the Alexandrov estimates for functions which are locally semiconvex/semiconcave or both. 
	%In this context, an \note{Save this comment for us?} interesting question is to obtain optimal constants, which will be treated in a successive paper. 
	Section \ref{sec:DME} presents the key result on determinant majorization for $I$-central polynomial operators  and the associated operator theory. Section \ref{sec:ABP_poly} presents the proof of the main result on Alexandrov estimates for $I$-central G\aa rding-Dirichlet operators, but with the proof of the crucial technical semiconvex approximation result of Lemma \ref{lem:QCA} later, in Appendix \ref{sec:App_B}. In Section \ref{sec:examples}, we briefly illustrate 
	the wealth of $I$-central G\aa rding-Dirichlet operators (for which the main theorem (Theorem \ref{thm:ABP_P1} applies). Finally, in Section \ref{sec:NG_Ops}, we give an extension of the main result to $I$-central homogeneous polynomial operators which are not necessarily G\aa rding operators in the sense of Definition \ref{defn:ICGD_poly}. The key point will be the variant the determinant majorization result (see Theorem \ref{thm:DME:NG}) and some representative examples and further questions will be discussed. Appendix \ref{sec:App_A} reviews the necessary operator theory and potential theory for inhomogeneous equations $G(D^2 u) = f$ as considered in this paper and beyond.

	\section{The Alexandrov estimate for locally semiconvex functions}\label{sec:ABP-lqc}
	
	In this section, we recall the basic pointwise estimate of Alexandrov type on the supremum of locally semiconvex functions, together with some easy corollaries for estimates on the infimum and oscillation. As noted in the introduction, the supremum estimate gives control on the growth from the boundary by quantifying the effect of ``concavity'' in the interior. There is a dual result on the decay from the boundary by quantifying the effect of ``convexity''. We remark that classically, one considers $u \in C^2(\Omega) \cap C(\overline{\Omega})$, which then by mollification extends to $W^{2,n}(\Omega) \cap C(\overline{\Omega})$. In \cite{PR25}, we prove this for $u \in \USC(\overline{\Omega})$ which is locally semiconvex. Before stating the results, we recall some basic notions. First, we recall the notion of (local) semiconvexity.
	
	\begin{defn} \label{defn:sc}
		A function $u\colon C\to \R$ is said to be {\em $\lambda$-semiconvex} on a convex set $C \subset \R^n$ if there exists a real number $\lambda \geq 0$ such that the function $u + \frac{\lambda}{2}|\cdot|^2$ is convex on $C$.
		
		We say that $u$ is {\em locally semiconvex} on an open subset $\Omega$ if for every $x\in \Omega$ there exists a ball $B\subset \Omega$ with $x\in B$ such that $u$ is $\lambda$-semiconvex on $B$, for some nonnegative real number $\lambda = \lambda(x)$. 
	\end{defn}
	
	Notice that there is a natural notion of {\em semiconcavity} for functions $v$ that is equivalent to $u :=-v$ being semiconvex. 
	
	The sets which quantify concavity and convexity in Alexandrov-type estimates are contained in the following definition. 
	
	\begin{defn} \label{defn:UC_LC} Given a function $u: \Omega \to \R$ on an open set $\Omega$ we say that $x \in \Omega$ is a {\em point of flat upper contact for $u$ at $x$} if there exists $p \in \R^n$ such that
		\begin{equation}\label{FUCP}
		u(y) \leq u(x) + \langle p, y - x \rangle, \quad \forall \, x \in \Omega;
		\end{equation}
		that is, $u$ admits a supporting hyperplane from above at the point $x$. The set of all flat upper contact points is denoted by $\Gamma^{+}(u)$ and call it the {\em upper contact set for $u$}.
		
		Similarly for the {\em lower contact set} $\Gamma^{-}(u)$ of flat lower contact points, where one reverses the inequality in \eqref{FUCP}.
	\end{defn}
	
	We are now ready to state the locally semiconvex version of Alexandrov's estimate. For a proof, we refer to Theorem 3.36 of \cite{PR25} which makes use of the {\em Area Formula for gradients of locally semiconvex functions} (see, for example, Theorem 3.35 of \cite{PR25}).
	\begin{thm}[Alexandrov's estimate for locally semiconvex functions] \label{thm:ABP_lqc}
		Let $\Omega \subset \R^n$ be open and bounded. If $u\in\USC(\overline\Omega)$ is locally semiconvex, then 
		\begin{equation} \label{ABP_lqc}
		\sup_{\Omega} u \leq \sup_{\partial \Omega} u + \frac{\mathrm{diam}(\Omega)}{|B_1|^{1/n}} \left( \int_{	E^{+}(u)} |\det \,  D^2u(x)| \, dx \right)^{1/n}.
		\end{equation}
	\end{thm}
	Here we have used the following notations. $|B_1|$ is the Lebesgue measure of the unit ball in $\R^n$ and 
	\begin{equation}\label{E}
	E^{+}(u) := \Gamma^{+}(u) \cap \Diff^2(u) \subset \Omega
	\end{equation}
	is the set of all flat global upper contact points at which $u$ is twice differentiable.

	A few remarks about the Alexandrov estimate are worth noting.  
	
	\begin{rem}\label{rem:ABP} Notice that:
		\begin{itemize}
			\item[(a)] In \eqref{ABP_lqc}, one can replace $	E^{+}(u)$ with the (a priori) larger set $\Gamma^{+}(u) $ since $\Diff^2(u)$ is of full measure in $\Omega$ for $u$ locally semiconvex. 
			\item[(b)] One has that $D^2u \leq 0$ on $E^{+}(u)$ and hence $|\det D^2u(x)| = \det(-D^2u(x))$ for each $x \in E^{+}(u)$.
			\item[(c)] In this setting, the estimate \eqref{ABP_lqc} is equivalent to the (a priori) weaker estimate which replaces $\sup_{\partial \Omega} u$ with $\sup_{\partial \Omega} u^+$, as one sees by applying the weaker estimate to $\bar{u} := u - \inf_{\partial \Omega}u$ which is non-negative on $\partial \Omega$. This fact is the first step in the proof in \cite{PR25}.
			\item[(d)] The upper bound  \eqref{ABP_lqc} is interesting only when $u$ is {\bf not subffine} on $\Omega$. We recall that $u \in \USC(\Omega)$ is subaffine if it satisfies the following comparison principle for all affine functions $a$ on $\Omega$:
			\begin{equation}\label{subaffine}
			u \leq a \ \ \text{on} \ \ \partial \Omega \ \ \Leftrightarrow \ \ u \leq a \ \ \text{on} \ \ \Omega.
			\end{equation}
			Such functions are characterized by saying that $u$ is a viscosity subsolution of the maximal eigenvalue operator (see, for example \cite{HL09}). Since the constant function $a:= \sup_{\partial \Omega} u$ is affine, by \eqref{subaffine}, one has that subaffine functions $u$ satisfy the standard maximum principle and thus the (a priori) stronger estimate 
			\begin{equation}\label{MP_SA}\tag{MP}
			\sup_{\Omega} u \leq \sup_{\partial \Omega} u.
			\end{equation}
			Hence, in the interesting (non subaffine) cases, one can view the Alexandrov upper bound \eqref{ABP_lqc} as a {\bf {\em weakened maximum principle}} estimate, with an {\bf {\em additive error term}}. 
		\end{itemize}
	\end{rem}
	
	Next, a pair of easy corollaries.
	
	\begin{cor}[Alexandrov's lower estimate for locally semiconcave functions] Let $\Omega \subset \R^n$ be open and bounded. If $v \in \LSC(\overline{\Omega})$ which is locally semiconcave on $\Omega$, then 
		\begin{equation} \label{ABP_lqc_dual}
		\inf_{\Omega} v \geq \inf_{\partial \Omega} v - \frac{\mathrm{diam}(\Omega)}{|B_1|^{1/n}} \left( \int_{	E^{-}(v)} |\det D^2v(x)| \, dx \right)^{1/n},
		\end{equation}	
		where $E^{-}(v):= \Gamma^{-}(v) \cap \Diff^2(v)$ is the (analogous) set of all flat global lower contact points at which $v$ is twice differentiable. 
	\end{cor}
	
	To see this, just apply the upper estimate \eqref{ABP_lqc} to $u:= -v$ and make the obvious transcription between sup/inf and the upper/lower contact points. Combining the upper and lower Alexandrov estimate immediately yields the following result.
	
	\begin{cor}[Alexandrov's oscillation estimate for $C^{1,1}$ functions] Let $\Omega \subset \R^n$ be open and bounded. If $w \in C(\overline{\Omega}) \cap C^{1,1}(\Omega)$, then 
		\begin{equation} \label{ABP_osc}
		\osc_{\Omega} w \leq \osc_{\partial \Omega} w + 2^{1 - 1/n}\, \frac{\mathrm{diam}(\Omega)}{|B_1|^{1/n}} \left( \int_{	E^{+}(w) \cup E^{-}(w)} |\det D^2w(x)| \, dx \right)^{1/n},
		\end{equation}	
		where the oscillation is $\osc_{\Omega}(w) := \sup_{\Omega}(w) - \inf_{\Omega}(w)$ and similarly for the boundary oscillation.
	\end{cor}
	
	To see this, since $w \in \USC(\overline{\Omega}) \cap \LSC(\overline{\Omega})$ is both locally semiconvex and locally semiconcave (see, for example,  Proposition 2.5 of \cite{PR25}), just take the difference of the upper and lower bounds \eqref{ABP_lqc} and \eqref{ABP_lqc_dual} and make use of the convexity of the function
	$\psi(t):=t^{1/n}$ for real $t \geq 0$.
	
	\section{Polynomial operators and determinant majorization}\label{sec:DME}
	
	In this section, we present the necessary background on $I$-central G\aa rding-Dirichlet polynomials $\mathfrak{g}$ as defined in Definition \ref{defn:ICGD_poly}. In addition to the crucial determinant majorization, we will recall important facts about the operator theory of the differential operators that they define. 
	
	The following result is Theorem 1.3 of \cite{HL24}. 
	
	\begin{theorem}\label{thm:DME} If $(\mathfrak{g}, \Gamma)$ is an $I$-central G\aa rding-Dirichlet pair with $\mathfrak{g}$ of degree $N$ on $\cS(n)$, then the determinant is majorized:
		\begin{equation}\label{DME}
		\mathfrak{g}(A)^{1/N} \geq 	\mathfrak{g}(I)^{1/N} \left( {\rm det} A \right)^{1/n}, \quad \forall \, A \geq 0 \ \text{in} \ \cS(n).  
		\end{equation}
	\end{theorem}
	
	In preparation for the combination of Theorem \ref{thm:DME} with the Alexandrov estimates of the previous section, a few remarks concerning the operator theory of $\mathfrak{g}$ acting on the Hessian of functions will be useful. What we have to say is by now %\note{Added last sentence, OK?} 
	well-known and hence we will be brief. Readers unfamiliar with the concepts introduced might wish to consult the monograph \cite{CHLP20} as well as Appendix \ref{sec:App_A}.
	
	\newpage
	\begin{rem}\label{rem:DGOps1}[Degenerate ellipticity and admissibility constraints]
		Notice that:
		\begin{itemize}
			\item[(a)] Since $\mathfrak{g}$ is a {\bf {\em G\aa rding-Dirichlet polynomial}} of degree $N$  with closed G\aa rding cone $\overline{\Gamma}$, one knows that (see Definition \ref{defn:ICGD_poly} (2))
			\begin{equation}\label{Pmono}
			\overline{\Gamma} + \cP \subset \overline{\Gamma},
			\end{equation}
			which is equivalent to $\cP \subset \overline{\Gamma}$. The property \eqref{Pmono} implies that $\overline{\Gamma}$ is a pure second order subequation, which is also a cone with vertex at the origin.
			\item[(b)] The operator defined by $\mathfrak{g}$ is {\bf {\em degenerate elliptic}} when restricted to the closed cone $\overline{\Gamma}$; that is,
			\begin{equation}\label{gDE}
			\mathfrak{g}(A + P) \geq \mathfrak{g}(A), \quad \forall \, A \in \overline{\Gamma}, P \geq 0.
			\end{equation}
			\item[(c)] Degenerate ellipticity is essential for the theory of viscosity solutions for equations involving $\mathfrak{g}$ and hence restricting the domain of $\mathfrak{g}$ to $C(\overline{\Gamma})$ is necessary.  The correct notion of viscosity subsolutions $u$ to an equation like
			\begin{equation}\label{NHE}
			\mathfrak{g}(D^2 u) = f(x), \quad x \in \Omega
			\end{equation}
			involves $\overline{\Gamma}$ as an  {\bf {\em admissibility constraint}}. This will be made precise in Definition \ref{defn:DGsolns} and implies that the natural notion of subsolutions to \eqref{NHE} must also be $\overline{\Gamma}$-subharmonic (their Hessians belong to $\overline{\Gamma}$ in the viscosity sense). They are not necessarily convex since $\cP \subset \overline{\Gamma}$.
			%More precisely, $u \in \USC(\Omega)$ is a {\em ($\overline{\Gamma}$-admissible viscosity) subsolution to \eqref{NHE}} if for every $x \in \Omega$ 
			%\begin{equation}\label{subsoln}
			%D^2 \varphi (x) \in \overline{\Gamma} \quad \text {and} \quad  \mathfrak{g}(D^2 \varphi(x)) \geq f(x)
			%\end{equation}
			%for each {\em upper test function} $\varphi$ for $u$ at $x$ ($\varphi$ is $C^2$ and satisfies $u \leq \varphi$ near $x$ with $u(x) = \varphi(x)$).  Hence $\overline{\Gamma}$-admissible subsolutions need only be $\overline{\Gamma}$-subharmonic (and hence not necessarily convex since $\cP \subset \overline{\Gamma}$).
		\end{itemize}
	\end{rem}
	
	\begin{rem}\label{rem:DGOps2}[Tameness and the correspondence principle] Notice also:
		\begin{itemize}
			\item[(a)] The operator $\mathfrak{g} \in C(\overline{\Gamma})$ is {\bf {\em tame}} (see Definition 2.3 of \cite{HL18b}) which in this case means: for every $\eta > 0$ there exists $c(\eta) > 0$ such that
			\begin{equation}\label{tame1}
			\mathfrak{g}(A + P) - \mathfrak{g}(A) \geq c(\eta), \quad \forall \, A \in \overline{\Gamma}, \ \forall \, P \geq \eta I,
			\end{equation}
			which is stronger than the degenerate ellipticity. 
			\item[(b)] In particular, if one uses $P = \eta I$ in \eqref{tame1}, one finds a useful special case of the tameness estimate. In terms of the G\aa rding $I$-eigenvalues $\lambda_1^{\mathfrak{g}}(A) \ldots  \lambda_N^{\mathfrak{g}}(A)$ for $A \in \cS(n)$, we recall the expansions of \eqref{GMA_op}) for $\eta \in \R$ and $A \in \cS(n)$:
			\begin{equation}\label{GMA}
			\mathfrak{g}(A + \eta I) =  \prod_{k=1}^N \left( \eta + \lambda_k^{\mathfrak{g}}(A) \right) \quad \text{and} \quad 	\mathfrak{g}(A) =  \prod_{k=1}^N \lambda_k^{\mathfrak{g}}(A).
			\end{equation}
			Making use of \eqref{GMA} with $\eta > 0$ and $A \in \overline{\Gamma}$ (for which $\lambda_k^{\mathfrak{g}}(A) \geq 0$) one finds 
			%\note{this is essentially Proposition 6.11 of [InHom]}
			\begin{equation}\label{tame2}
			\mathfrak{g}(A + \eta I) - \mathfrak{g}(A) = \mathfrak{g}(I) \left[  \prod_{k=1}^N \left( \eta + \lambda_k^{\mathfrak{g}}(A) \right) - \prod_{k=1}^N \lambda_k^{\mathfrak{g}}(A)  \right] \geq  \eta^N > 0.
			\end{equation}
			\item[(c)] The pair $(\mathfrak{g}, \overline{\Gamma})$ is a {\bf {\em compatible}} (pure second order) operator-subequation pair in the sense of Definition 11.1 of \cite{CHLP20}, where the compatibility (in this constrained case) means
			\begin{equation}\label{compatibility} 
			\partial \overline{\Gamma} = \{ A \in \overline{\Gamma}: \mathfrak{g}(A) = 0\}.
			\end{equation}
			$(\mathfrak{g}, \overline{\Gamma})$ is {\bf {\em $\cM$-monotone}} in the sense of Definition 11.2 of \cite{CHLP20} for the monotoncity cone subequation $\cM = \R \times \R^n \times \cP$, which in this case pure second order case reduces to 
			\begin{equation}\label{MM}
			\overline{\Gamma} + \cP \subset \overline{\Gamma} \quad \text{and} \quad \mathfrak{g}(A + P) \geq \mathfrak{g}(A), \ \ \forall \, A \in \overline{\Gamma}, P \in \cP.
			\end{equation}
			The strict monotonicity implied by the tameness of $\mathfrak{g}$ on $\overline{\Gamma}$ ensures that $\mathfrak{g}$ is {\bf {\em topologically tame}}; that is, for each admissible level $c \in \mathfrak{g}(\overline{\Gamma})$ the level set $\mathfrak{g}(c) := \{A \in \overline{\Gamma}: \mathfrak{g}(A) = c\}$ has non empty interior (see Theorem 11.10 of \cite{CHLP20}).  Hence, by Theorem 11.13 of \cite{CHLP20}, we have the {\bf {\em correspondence principle}} for admissible subsolutions/supersolutions for the constant coefficient equation $\mathfrak{g}(D^2u) = c$ for each admissible level $c \in \R$. This correspondence principle will be extended in Theorem \ref{thm:CPP} to allow for continuous functions $f$ in place of constants $c$.
		\end{itemize}
	\end{rem}

	\section{Alexandrov estimates for polynomial operators}\label{sec:ABP_poly}
	
	In this section, we present the main result of this paper. Assume that $(\mathfrak{g}, \Gamma)$ is an $I$-central G\aa rding-Dirichlet operator-subequation pair with $\mathfrak{g}$ of degree $N$ on $\cS(n)$.  On each bounded open set $\Omega \subset \R^n$ and for each $f \in C(\Omega)$ %\note{Slight generalization where we assume only $f$ continuous on $\Omega$. This is related to the improved presentation of fiberegulaity below that will be used for example \ref{exe:NGMA}.} 
	with $f \geq 0$, consider the following equation on $\Omega$:
	\begin{equation}\label{DGE}
	\mathfrak{g}(D^2h) = f \quad \mbox{with the admissibility constraint $D^2h \in \overline{\Gamma}$}.
	\end{equation}

	%\noindent{\bf Working note for Reese:} Maybe we will want to make a new section 3 (or split into two subsections here) which gives all of the needed background for the equation \eqref{DGE} and its associated potential theory. I sketched some of the elements below in order to make my proposed proof of the Alexandov estimate for dual subharmonics readable.
	
	Before stating and proving the Alexandrov estimates for \eqref{DGE}, we recall the notion of viscosity solutions with $\overline{\Gamma}$ as an admissibility constraint that was anticipated in Remark \ref{rem:DGOps1} (c). See \cite{CIL92} for the conventional viscosity theory where no admissibility constraints are needed. As in all viscosity formulations, one makes use of spaces of upper and lower test functions or jets. There are many equivalent formulations. Here we will make use of the spaces of {\em upper/lower contact jets} of a function $u$ at a point $ x \in \Omega$ defined as follows. 
	
	\begin{defn}\label{defn:UCJ} Let $u$ be a function defined on $\Omega \subset \R^n$.  We say that $\varphi$ is an {\em upper test function for $u$ at $x \in \Omega$} if
		\begin{equation}\label{UCF}
		\mbox{$\varphi$ is $C^2$ and $u \leq \varphi$ near $x$ \quad and  \quad $u(x) = \varphi(x)$}.
		\end{equation}
		The space of {\em upper contact jets for $u$ at $x$} is defined by
		\begin{equation}\label{UCJ}
		J^{2,+}_x u := \{ (\varphi(x), D \varphi(x), D^2 \varphi(x)): \ \varphi \ \text{is an upper test function for $u$ at $x$} \}.
		\end{equation}
		The spaces of {\em lower test functions for $u$ at $x \in \Omega$} and {\em lower contact jets 	$J^{2,-}_x u$  for $u$ at $x$} are defined similarly. It is enough to replace use $u \leq \varphi$ near $x$ with   $u \geq \varphi$ near $x$ in \eqref{UCF}.
	\end{defn}
	
	\begin{defn}\label{defn:DGsolns} With $\mathfrak{g}, f$ and $\Omega$ as above,
		\begin{itemize}
			\item[(a)] $u \in \USC(\Omega)$ is an {\em admissible subsolution of} \eqref{DGE} {\em on $\Omega$} if 
			\begin{equation}\label{sub}
			\forall \, x \in \Omega: \quad (r,p,A) \in J^{2,+}_x u \quad \Rightarrow \quad A \in \overline{\Gamma} \ \text{and} \ \mathfrak{g}(A) \geq f(x).
			\end{equation}
			\item[(b)] $v \in \LSC(\Omega)$ is an {\em admissible supersolution of} \eqref{DGE} {\em on $\Omega$} if 
			\begin{equation}\label{super}
			\forall \, x \in \Omega: \quad (r,p,A) \in J^{2,-}_x u \quad \Rightarrow \quad \mbox{$A \not\in \overline{\Gamma}$ or  [$ A \in \overline{\Gamma}$ and $\mathfrak{g}(A) \leq f(x)$]}.
			\end{equation}
			\item[(c)] $h \in C(\Omega)$ is an {\em admissible solution of} \eqref{DGE} {\em on $\Omega$} if both (a) and (b) hold with $u = h$ and $v = h$ respectively.
		\end{itemize}
	\end{defn}
	%We have used the simplified terms {\em subsolutions, supersolutions} and {\em solutions} of \eqref{DGE} on $\Omega$ for what are in fact {\em $\overline{\Gamma}$-admissible subsolutions, supersolutions} and {\em solutions} of the equation $\mathfrak{g}(D^2h) = f$ on $\Omega$.
	
	Denoting by $\cJ^2 := \R \times \R^n \times \cS(n)$ the vector space of $2$-jets, we have an equivalent potential theoretic formulation in terms of the (pure second order) constraint set with fibers
	\begin{equation}\label{Fsub}
	\F_x := \{(r,p,A) \in  \cJ^2: \ A \in \overline{\Gamma} \ \text{and} \ \mathfrak{g}(A) - f(x) \geq 0\}, \quad x \in \Omega
	\end{equation}
	and its dual constraint set with fibers
	\begin{equation}\label{Fdualsub}
	\wt{\F}_x := \{(r,p,A) \in  \cJ^2: \ \mbox{$-A \not\in \overline{\Gamma}$ or  [$-A \in \overline{\Gamma}$ and $\mathfrak{g}(-A) - f(x) \leq 0$]}\}, \quad x \in \Omega.
	\end{equation}
	
	\begin{rem}[Equivalent potential theoretic formulation]\label{rem:correspondence}
		Consider the constrained case operator-subequation pair defined by 
		\begin{equation}\label{compatible_pair}
		F(x,r,p,A):= \mathfrak{g}(A) - f(x) \quad \text{and} \quad \cG := \R \times \R^n \times \overline{\Gamma}.
		\end{equation}
		whose $\cG$-admissible subsolutions (supersolutions) are precisely the subsolutions (supersolutions) of Definition \ref{defn:DGsolns}. The natural candidate for a subequation associated to the pair $(F, \cG)$ is defined by the {\bf {\em  correspondence relation}}
		\begin{equation}\label{CR}
		\F_x := \{ (r,p,A) \in \cG: \ F(x,r,p,A) \geq 0 \}, \quad x \in \Omega,
		\end{equation}
		which is exactly the constraint \eqref{Fsub}. 
		
		Moreover, the pair $(F, \cG)$ satisfies two important properties.
		\begin{enumerate}
			\item $(F, \cG)$ is {\bf {\em $\cM$-monotone}} for the monotonicity cone subequation $\cM(\cP) = \R \times \R^n \times \cP$; this means that for each $x \in \Omega$:
			$$
			\cG + \cM(\cP) \subset \cG \quad \text{and} \quad F(x,J + J') \geq F(x,J), \forall \, J \in \cG, J' \in \cM(\cP).
			$$
			This fact follows from the degenerate ellipticity of $\mathfrak{g}$ on $\overline{\Gamma}$ and the fact that $F$ is of pure second order.
			\item  $(F, \cG)$ is {\bf {\em fiberegular}}  on $\Omega$. By defintion, this means that $\cG$ is fiberegular and (since the pair is $\cM$-monotone \footnote{See Definition 3.1 of \cite{CPR23} for the general notion of fiberegularity of for a subequation defined by \eqref{CR} and its equivalent reformulations such as \eqref{fibereg1} for $\cM$-monotone subequtions in Proposition 3.2 of \cite{CPR23}}) for every $\Omega' \subset \subset \Omega$ and for every $\eta > 0$ there exists $\delta = \delta(\Omega', \eta) > 0$ such that
			\begin{eqnarray}\label{fibereg1}
			\forall \, x,y \in \Omega' \ \text{with} \ |x - y| < \delta \quad \text{and} \quad \forall \, (r,p,A) \in \cG: \nonumber \\
			F(x,r,p,A) \geq 0 \ \ \Rightarrow \ \ F(y,r,p,A + \eta I) \geq 0. 
			\end{eqnarray}
			We will refer to $\delta = \delta(\Omega', \eta) > 0$ as the {\em modulus of continuity of the $\cM(\cP)$-monotone pair $(F, \cG)$} on $\Omega' \subset \subset \Omega$.
			
			The fiberegularity of $\cG$ is automatic because $\cG$ has constant fibers. The implication \eqref{fibereg1} follows by using the unform continuity of $f$ on $\Omega'$ and the tameness \eqref{tame2} of the pair $(\mathfrak{g}, \overline{\Gamma})$. More precisely, by the uniform continuity of $f$, for each $\veps > 0$ there exists $\delta_{f,\Omega'} = \delta_{f,\Omega'}(\veps) > 0$ such that
			\begin{equation}\label{fUC}
			|f(x) - f(y) | < \veps \quad \forall \, x,y \in \Omega' \ \text{with} \ |x-y| < \delta_{f,\Omega'}(\veps)
			\end{equation}
			Combining this with the tameness \eqref{tame2} of $(\mathfrak{g}, \overline{\Gamma})$ one has: for each $\eta > 0$ and for each $(r,p,A) \in \cG = \R \times \R^n \times \overline{\Gamma}$ with $G(A) - f(x) \geq 0$
			\begin{eqnarray*}
				F(y,r,p,A + \eta I) &=& 	[\mathfrak{g}(A + \eta I) - \mathfrak{g}(A)] + [G(A) - f(x)] +  [f(y) + f(x)] \\
				&\geq&  \eta^N + 0 - \veps, \quad 
			\end{eqnarray*}
			which holds $\forall \, x,y \in \Omega' \ \text{with} \ |x - y| < \delta_{f,\Omega'}(\veps)$. This gives \eqref{fibereg1} with modulus of continuity
			\begin{equation}\label{MOC_fibereg}
			\delta(\Omega',\eta) = \delta_{f,\Omega'}\left(  \eta^N \right).
			\end{equation}
			
			Notice that for $\eta$ fixed, \eqref{fibereg1} holds for all $\delta \in (0, \delta(\Omega', \eta))$ where $\delta(\Omega', \eta) \to 0^+$ as $\eta \to 0^+$ and hence $\delta$ is indeed a modulus of continuity for the pair $(F, \cG)$.
		\end{enumerate}
		From points (1) and (2), Theorem 7.11 of \cite{CPR23} shows that the constraint set defined by \eqref{Fsub} is an $\cM(\cP)$-monotone subequation which is fiberegular in the sense that
		\begin{equation}\label{fibereg2}
		\forall \, \eta > 0: \quad\F_x + \eta (0,0,I) \subset \F_y \quad \forall \, x,y \in \Omega \ \text{with} \ |x - y| < \delta = \delta(\Omega, \eta),
		\end{equation}
		which by duality carries over to $\wt{\F}$ defined in \eqref{Fdualsub}, with the same modulus of continuity $\delta$ as in \eqref{fibereg2}.
		
		Finally, $(F, \cG)$ and $\F$ also satisfy the {\bf {\em compatibility condition}} 
		\begin{equation}\label{CC}
		\Int \, \F = \{(x,r,p,A) \in \cG: F(x,r,p,A) > 0\}.
		\end{equation}
	\end{rem}
	
	Hence by applying Theorem 7.8 of \cite{CPR23}, we have the following version of the correspondence principle. See \cite{HP25} for a recent survey on the history and importance of the correspondence principle.
	
	\begin{thm}[Correspondence principle for inhomogeneous polynomial equations]\label{thm:CPP}
		Suppose that $\mathfrak{g}$ is an $I$-central G\aa rding-Dirichlet polynomial on $\cS(n)$ and that $f \in C(\Omega)$. One has the following statements.
		\begin{itemize}
			\item[(a)] For each $u \in \USC(\Omega)$,
			\begin{equation*}\label{CP1}
			\mbox{$u$ is an admissible subsolution of \eqref{DGE} on $\Omega \Leftrightarrow  u$ is $\F$-subharmonic on $\Omega$},
			\end{equation*}
			where $\F$ is the subequation defined by \eqref{Fsub}.
			\item[(a)] For each $v \in \LSC(\Omega)$
			\begin{equation*}\label{CP2}
			\mbox{$v$ is an admissible supersolution of \eqref{DGE} on $\Omega \Leftrightarrow  u$ is $\F$-superharmonic on $\Omega$},
			\end{equation*}
			which by duality, is equivalent to 
			\begin{equation*}\label{CP3}
			\mbox{$v$ is an admissible supersolution of \eqref{DGE} on $\Omega \Leftrightarrow  w := -v$ is $\wt{\F}$-subharmonic on $\Omega$},
			\end{equation*}
			where $\wt{\F}$ is the subequation defined by \eqref{Fdualsub}.
		\end{itemize}
	\end{thm}
	
	%Exploiting this correspondence, in order to simplify nomenclature, we will refer to subsolutions $u$, supersolutions $v$ and solutions $h$ of \eqref{DGE} as {\em subharmonics, superharmonics} and {\em harmonics}, respectively and use the term {\em dual subharmonic} for the functions $w$ in \eqref{CP3}.

	We are now ready for the main result of this paper.
	
	\begin{thm}[Alexandrov estimates for $I$-central polynomial operators]\label{thm:ABP_P1} Suppose that $(\mathfrak{g}, \Gamma)$ is an $I$-central G\aa rding-Dirichlet pair with $\mathfrak{g}$ of degree $N$ on $\cS(n)$ and supoose that $f \in C(\R^n)$ with $f \geq 0$. For each $\Omega \subset \R^n$ open and bounded one has the following Alexandrov estimates.
		\begin{itemize}
			\item[(a)] If  $u \in \USC(\overline{\Omega})$ is an admissible subsolution of \eqref{DGE} on $\Omega$ (equivalently, $u$ is $\F$-subharmonic on $\Omega$ for $\F$ defined by \eqref{Fsub}) then
			\begin{equation}\label{ABPDG_sub}
			\sup_{\Omega} u \leq \sup_{\partial \Omega} u;
			\end{equation}
			that is, the maximum principle ${\rm (MP)}$ holds.
			\item[(b)] If  $v \in \LSC(\overline{\Omega})$ is an admissible supersolution of \eqref{DGE} on $\Omega$ (equivalently, $v$ is $\F$-superharmonic on $\Omega$) then
			\begin{equation}\label{ABPDG_super}
			\inf_{\Omega} v \geq \inf_{\partial \Omega} v - \frac{\mathrm{diam}(\Omega)}{|B_1|^{1/n} \mathfrak{g}(I)^{1/N}} || f ^{1/N}||_{L^n(\Omega)}.
			\end{equation}
			Equivalently, $w:=-v \in \USC(\Omega)$ is $\wt{\F}$-subarmonic on $\Omega$ with $\wt{\F}$ defined by \eqref{Fdualsub} and satisfies 
			\begin{equation}\label{ABPDG_dual}
			\sup_{\Omega} w \leq \sup_{\partial \Omega} w + \frac{\mathrm{diam}(\Omega)}{|B_1|^{1/n}\mathfrak{g}(I)^{1/N}} || f ^{1/N}||_{L^n(\Omega)}.
			\end{equation}
			\item[(c)] If  $h \in C(\overline{\Omega})$ is an admissible solution of \eqref{DGE} in $\Omega$ (equivalently, $h$ is $\F$-harmonic on $\Omega$) then
			\begin{equation}\label{ABPDG_osc}
			\osc_{\Omega} h \leq \osc_{\partial \Omega} h + \frac{\mathrm{diam}(\Omega)}{|B_1|^{1/n}\mathfrak{g}(I)^{1/N}} || f ^{1/N}||_{L^n(\Omega)}.
			\end{equation}
		\end{itemize}
	\end{thm}
	%\newpage

	\begin{proof}
		For the maximum principle \eqref{ABPDG_sub} of part (a), if $u$ is an admissible subsolution on $\Omega$, by definition, for each $x \in \Omega$ one has 
		$$
		J^{2,+}_x u \subset \F_x = \{(r,p,A): A \in \overline{\Gamma} \ \text{and} \ \mathfrak{g}(A) - f(x) \geq 0\} \subset  \{(r,p,A): A \in \overline{\Gamma}\};
		$$
		that is, $u$ is $\Gamma$-subharmonic on $\Omega$ for the constant coefficient pure second order subequation cone $\overline{\Gamma}$. Hence, the validity of (MP) reduces to the following fact. 
		\begin{lem}\label{lem:Gamma_MP}
			For every $I$-central pair $(\mathfrak{g}, \Gamma)$, the G\aa rding cone $\Gamma$ satisfies 
			\begin{equation}\label{GD_containment}
			\Gamma \subset \{ A \in \cS(n): \ {\rm tr} A > 0 \} \quad \text{and hence} \quad \overline{\Gamma} \subset \{ A \in \cS(n): \ {\rm tr} A \geq 0 \}. 
			\end{equation}
			Consequently, every $u \in \USC(\overline{\Omega})$ which is $\overline{\Gamma}$-subharmonic on a bounded open subset $\Omega \subset \R^n$ is a classical (laplacian) subharmonic function on $\Omega$, and hence, satisfies the maximum principle \eqref{ABPDG_sub}.
		\end{lem} 
		\begin{proof}
			It suffices to prove the first inclusion in \eqref{GD_containment}, which is proven in Proposition 3.1 (2) of \cite{HL24}. For the convenience of the reader, we give the main ideas of the proof. First, for any G\aa rding-Dirichlet polynomial $\mathfrak{g}$ of degree $N$ on $\cS(n)$, each {\em radial derivative}
			\begin{equation}\label{GD1}
			\mathfrak{g}^{(\ell)}(A) :=  \frac{d^{(\ell)}}{dt^{(\ell)}} \left(\mathfrak{g}(tI + A) \right)_{|t = 0}, \quad \forall \, A \in \Gamma, \forall \, \ell = 1, \ldots N - 1
			\end{equation}
			is a G\aa rding-Dirichlet polynomial of degree $N - \ell$ and the associated G\aa rding cones are nested
			$$
			\Gamma_{\mathfrak{g}} \subset 	\Gamma_{\mathfrak{g}^{(1)}} \cdots \Gamma_{\mathfrak{g}^{(N -1)}} = \Int \, \mathcal{H},
			$$
			where $\mathcal{H}$ is a half-space. 
			
			Then, if $\mathfrak{g}$ is $I$-central, one easily verifies that $\mathcal{H}  = \{ A \in \cS(n): \ {\rm tr} A > 0 \}$.
		\end{proof}

		The oscillation estimate \eqref{ABPDG_osc} of part (c) clearly follows from the upper and lower estimates of parts (a) and (b). In fact, one uses $u = h$ in \eqref{ABPDG_sub} and $w = -h$ in \eqref{ABPDG_super} and then subtracts the lower estimate from the upper estimate.
		
		It remains to prove only one of the equivalent estimates \eqref{ABPDG_super} or \eqref{ABPDG_dual} of part (b). We will give a direct proof of \eqref{ABPDG_dual} for dual subharmonics; that is, for $w \in \USC(\overline{\Omega})$ which are $\wt{\F}$-subharmonic in $\Omega$ with $\wt{\F}$ defined by \eqref{Fdualsub}. The idea of the proof is to combine the Alexandrov estimate of Theorem \ref{thm:ABP_lqc} with the determinant majorization estimate of Theorem \ref{thm:DME} along a suitable sequence of locally semiconvex functions which are $\wt{\F}$-subharmonic and approximate $w \in \USC(\overline{\Omega}) \cap \wt{\F}(\Omega)$. 
		
		The proof will involve four steps, but first some preliminary observations are helpful. As we have already noted, the variable coefficient pure second order subequation $\F$ defined by
		$$
		\F_x := \{(r,p,A) \in  \cJ^2: \ A \in \overline{\Gamma} \ \text{and} \ \mathfrak{g}(A) - f(x) \geq 0\}, \quad x \in \Omega
		$$
		is a fiberegular and $\cM$-monotone subequation for the monotonicity cone subequation $\cM = \cM(\cP) := \R \times \R^n \times \cP$. By duality, these two facts also hold for the dual subequation with fibers
		$$
		\wt{\F}_x := \{(r,p,A) \in  \cJ^2: \ \mbox{$-A \not\in \overline{\Gamma}$ or  [$-A \in \overline{\Gamma}$ and $\mathfrak{g}(-A) - f(x) \leq 0$]}\}, \quad x \in \Omega.
		$$
		Since $\wt{\F}$ has variable coefficients, translates and hence sup-convolutions of $\wt{\F}$-subharmonics $w$ need not necessarily be $\wt{\F}$-subharmonic; however, for \underline{bounded} $w$ we have a good approximation procedure as will be described in Step 2.
		
		\noindent \underline{Step 1:} (Reduction to $w$ bounded) First notice that $w \in \USC(\overline{\Omega}) \cap \wt{\F}(\Omega)$ is bounded from above by the upper semicontinuity of $w$ on $\overline{\Omega}$ compact. Without loss of generality, we can suppose that $w$ is also bounded from below. Indeed, suppose that for each $w \in \USC(\overline{\Omega}) \cap \wt{\F}(\Omega)$ which is bounded below, one has the desired estimate
		\begin{equation}\label{recall_Alex_dualsub}
		\sup_{\Omega} w \leq \sup_{\partial \Omega} w + \frac{\mathrm{diam}(\Omega)}{|B_1|^{1/n}} || f ^{1/N}||_{L^n(\Omega)}.
		\end{equation}
		If $w$ is not bounded from below, then we can truncate $w$ from below by defining for each $m \in \N$:
		\begin{equation}
		w_m(x) := \max \{ w(x), -m\}, \quad \forall \, x \in \overline{\Omega}.
		\end{equation}
		For each $m \in \N$, one has $w_m \in \USC(\overline{\Omega}) \cap \wt{\F}(\Omega)$ and is bounded from below. The semicontinuity and boundedness follow easily from the definition of $w_m$ as does the claim of being $\wt{\F}$-subharmonic since clearly all constant functions have zero Hessian (and hence lie in $\overline{\Gamma}$) and $\mathfrak{g}(D^2(-m)) - f(x) = -f(x) \leq 0$ by the non-negativity of $f$. Then use the easy fact that the maximum of two subharmonics is always a subharmonic (see Proposition D.1 (B) of \cite{CHLP20}). Applying the estimate \eqref{recall_Alex_dualsub} to the sequence $\{w_m\}$ yields 
		$$
		\sup_{\Omega} w_m \leq \sup_{\partial \Omega} w_m + \frac{\mathrm{diam}(\Omega)}{|B_1|^{1/n}} || f ^{1/N}||_{L^n(\Omega)}, \quad \forall \, m \in \N.
		$$
		However, since $w \in \USC(\overline{\Omega})$ has a finite supremum on both $\Omega$ and $\partial \Omega$, one has
		$$
		\sup_{\Omega} w_m = \sup_{\Omega} w \quad \text{and} \quad \sup_{\partial \Omega} w_m = \sup_{\partial \Omega} w, \quad \text{for all large $m$}, 
		$$
		which completes the claim.
		
		\noindent \underline{Step 2:} (The approximation procedure and its properties) Fix any function $\psi \in \USC(\overline{\Omega}) \cap C^2(\Omega)$ which is strictly $\cM$-subharmonic on $\Omega$ for the monotonicity cone subequation $\cM =\cM(P)$ associated to $\wt{\F}$. This just means that $\psi$ is striclty convex. For example, the non-negative quadratic function $\psi := \frac{1}{2} | \cdot|^2$ will do for any bounded $\Omega$. Then for each $\veps > 0$ and each $\eta > 0$ define the perturbed sup-convolution approximations
		\begin{equation}\label{approx1}
		w_{\eta}^{\veps}(x) := w^{\veps}(x) + \eta \psi(x), \quad \text{with} \ \psi(\cdot) = \frac{1}{2} |\cdot|^2 \quad \forall \, x \in \Omega,
		\end{equation}
		where 
		\begin{equation}\label{approx2}
		w^{\veps}(x) := \sup_{y \in \Omega} \left( w(y) - \frac{1}{2 \veps} |y - x|^2 \right), \quad \forall \, x \in \Omega,
		\end{equation}
		is the {\em $\veps$-sup-convolution of $w$}. One knows that	$w^{\veps}$ is $\frac{1}{\veps}$-semiconvex and $w^{\veps} \searrow w$ on $\Omega$ (see, for example, Theorem 4.10 of \cite{PR25}). However, for variable coefficient subeqations like $\wt{\F}$, the regularization $w^{\veps}$ might not be $\wt{\F}$-subharmonic on $\Omega$. The fiberegularity of $\wt{\F}$ ensures that (in a uniform sense) small quadratic perturbations like \eqref{approx1} will be $\wt{\F}$-subharmonic (away from the boundary of $\Omega$) and the (strict) convexity of $\psi$ ensures that $\frac{1}{\veps}$-semiconvexity of $w^{\veps}$ passes to the perturbations $w_{\eta}^{\veps} = w^{\veps} + \eta \psi$. The following result is a special case of Theorem 4.2 of \cite{CPR23}, adapted to our pure second order situation and stated directly for (dual) $\wt{\F}$-subharmonics.
		For completeness, we will give a proof in Appendix \ref{sec:App_B}.
		
		\begin{lem}[Semiconvex approximation of dual subharmonics]\label{lem:QCA} Suppose that $\wt{\F}$ is an $\cM(\cP)$-monotone and fiberegular subequation (with modulus of continuity $\delta = \delta(\eta)$) on $\Omega$ open and bounded.  Suppose that $w \in \wt{\F}(\Omega)$ is bounded, with $|w| \leq M$ on $\Omega$. Then for each $\eta > 0$ there exists $\veps_{\ast} = \veps_{\ast}(\delta(\eta), M) > 0$ such that the perturbed sup-convolutions defined by \eqref{approx1} satisfy
			\begin{equation}\label{approx}
			w_{\eta}^{\veps} \in \wt{\F}(\Omega_{\delta}), \quad \forall \, \veps \in (0, \veps_{\ast}],
			\end{equation}
			where $\Omega_{\delta}:= \{ x \in \Omega: \ {\rm dist}(x, \partial \Omega) > \delta\}$.
		\end{lem} 
		
		%We note that the proof of this lemma makes use of the {\em uniform translation property} (see Theorem 3.3 of \cite{CPR23}). For $w \in \wt{\F}(\Omega)$ (not necessarily bounded), this property says that under the assumptions on $\wt{\F}$ and $\Omega$ in Lemma \ref{lem:QCA}: {\em for every $\eta > 0$ there exists $\delta = \delta(\psi,\eta) > 0$ such that the perturbed translations}
		%	\begin{equation}\label{UTP}
		%	w_{y, \eta}:= \tau_{y}w + \eta \psi \in \wt{\F}(\Omega_{\delta}), \quad \forall \, y \in B_{\delta}(0),
		%	\end{equation}
		%	{\em where $\tau_y w(\cdot) := w(\cdot - y)$}.
		%Moreover, the \note{maybe a discussion before of fiberegularity is in order} parameters $\delta$ and $\eta$ come from the fiberegularity of $\wt{\F}$, which are the same as those for $\F$. 
		
		\noindent \underline{Step 3:} (The Alexandrov estimate for the approximations $	w_{\eta}^{\veps}$ on $\Omega_{\delta}$)	We claim that: {\em for each $\eta > 0$, with  $\delta = \delta(\eta) > 0$ and $\veps_{\ast} = \veps_{\ast}(\delta, M) > 0$ as in Lemma \ref{lem:QCA}, for each $\veps \in (0, \veps_{\ast}]$ one has}
		\begin{equation}\label{Alex0}
		\sup_{\Omega_{\delta}} w_{\eta}^{\veps} \leq \sup_{\partial \Omega_{\delta}} w_{\eta}^{\veps} + \frac{\mathrm{diam}(\Omega_{\delta})}{|B_1|^{1/n} \mathfrak{g}(I)^{1/N}} \left( \int_{E^{+}(w_{\eta}^{\veps}) \cap \Omega_{\delta}} f(x)^{n/N} \, dx \right)^{1/n}. 
		\end{equation}
		
		To prove this claim, we begin by noting that for each $\eta > 0$, $w_{\eta}^{\veps}$ is $\frac{1}{\veps}$-semiconvex on $\Omega$ for each $\veps > 0$. Hence, for each $\delta > 0$, we can apply the Alexandrov estimate \eqref{ABP_lqc} of Theorem \ref{thm:ABP_lqc} on $\Omega_{\delta}$ to deduce
		\begin{equation}\label{Alex1}
		\sup_{\Omega_{\delta}} w_{\eta}^{\veps} \leq \sup_{\partial \Omega_{\delta}} w_{\eta}^{\veps} + \frac{\mathrm{diam}(\Omega_{\delta})}{|B_1|^{1/n}} \left( \int_{E^{+}(w_{\eta}^{\veps}) \cap \Omega_{\delta}} |\det \,  D^2w_{\eta}^{\veps}(x)| \, dx \right)^{1/n}. 
		\end{equation}
		Now, on $E^{+}(w_{\eta}^{\veps}) = {\rm Diff}^2(w_{\eta}^{\veps}) \cap \Gamma^+(w_{\eta}^{\veps})$  we have $D^2(w_{\eta}^{\veps}) \leq 0$ in $\cS(n)$ (since $E^{+}(w_{\eta}^{\veps})$ are flat upper contact points of twice differentiability) and hence
		\begin{equation}\label{Alex2}
		- D^2w_{\eta}^{\veps}(x)  \in \cP \subset \overline{\Gamma}, \quad \forall \, x \in E^{+}(w_{\eta}^{\veps}) \cap \Omega_{\delta}
		\end{equation}
		and
		\begin{equation}\label{Alex3}
		|\det \,  D^2w_{\eta}^{\veps}(x)| =  \det \, \left( - D^2w_{\eta}^{\veps}(x) \right) \quad \forall \, x \in E^{+}(w_{\eta}^{\veps}) \cap \Omega_{\delta}.
		\end{equation}
		By applying the semiconvex approximation of Lemma \ref{lem:QCA}, for each $\veps \in (0, \veps_{\ast}]$, $w_{\eta}^{\veps}$ is $\wt{\F}$-subharmonic on $\Omega_{\delta}$. Then using \eqref{Alex2} and the defintion of $\wt{\F}$ (see \eqref{Fdualsub}) we can conclude that
		\begin{equation}\label{Alex4}
		\mathfrak{g}(- D^2w_{\eta}^{\veps}(x) ) - f(x) \leq 0, \quad \forall \, x \in E^{+}(w_{\eta}^{\veps}) \cap \Omega_{\delta}.
		\end{equation}
		Since $- D^2w_{\eta}^{\veps}(x)  \in \cP$ for all $ x \in E^{+}(w_{\eta}^{\veps}) \cap \Omega_{\delta}$ (see \eqref{Alex2}), applying the determinant majorization \eqref{DME} of Theorem \ref{thm:DME} to the identity \eqref{Alex3} gives
		\begin{equation}\label{Alex5}
		|\det \,  D^2w_{\eta}^{\veps}(x)| = \det \, \left( - D^2w_{\eta}^{\veps}(x) \right) \leq \frac {\mathfrak{g}(- D^2w_{\eta}^{\veps}(x))^{n/N}}{\mathfrak{g}(I)^{n/N}}, \quad \forall \, x \in E^{+}(w_{\eta}^{\veps}) \cap \Omega_{\delta}.
		\end{equation}
		Finally, combining \eqref{Alex4} with \eqref{Alex5} yields
		\begin{equation}\label{Alex6}
		|\det \,  D^2w_{\eta}^{\veps}(x)| \leq \frac {f(x)^{n/N}}{\mathfrak{g}(I)^{n/N}}, \quad \forall \, x \in E^{+}(w_{\eta}^{\veps}) \cap \Omega_{\delta},
		\end{equation}
		which when inserted into \eqref{Alex1} gives the estimate \eqref{Alex0}.
		
		\noindent \underline{Step 4:} (Passing to the limit in $\veps, \delta$ and $\eta$ in \eqref{Alex0}) For simplicity, we will not track the dependence of the integration domain $E^{+}(w_{\eta}^{\veps}) \cap \Omega_{\delta}$ on the parameters $(\veps, \delta, \eta)$ during the limiting procedure. This is accomplished by exploiting monotonicity in the Lebesgue integral to weaken the estimate \eqref{Alex0} to 
		\begin{equation}\label{Alex7}
		\sup_{\Omega_{\delta}} w_{\eta}^{\veps} \leq \sup_{\partial \Omega_{\delta}} w_{\eta}^{\veps} + \frac{\mathrm{diam}(\Omega)}{|B_1|^{1/n} \mathfrak{g}(I)^{1/N}} ||f^{1/N}||_{L^n(\Omega)}.
		\end{equation}
		To simplify notation, we rewrite \eqref{Alex7} as
		\begin{equation}\label{Alex8}
		\sup_{\Omega_{\delta}} w_{\eta}^{\veps} \leq \sup_{\partial \Omega_{\delta}} w_{\eta}^{\veps} + C \quad \text{where} \quad C := \frac{\mathrm{diam}(\Omega)}{|B_1|^{1/n} \mathfrak{g}(I)^{1/N}} ||f^{1/N}||_{L^n(\Omega)}.
		\end{equation}
		
		We first consider the limit as $\veps \searrow 0$. For each $\eta > 0$ and $\delta = \delta(\eta) >0$ fixed, \eqref{Alex8} holds for every $\veps \in (0, \veps_{\ast}(M,\delta) ]$ and we know that as $\veps \searrow 0$
		\begin{equation}\label{conv1}
		w_{\eta}^{\veps} := w^{\veps} + \eta \psi \searrow w_{\eta} := w + \eta \psi, \quad \text{where the convergence is uniform on $\overline{\Omega_{\delta}}$}.
		\end{equation}
		Hence as $\veps \searrow 0$, we have
		$$
		0 \leq \sup_{\Omega_{\delta}} ( w^{\veps} + \eta \psi) - \sup_{\Omega_{\delta}} ( w + \eta \psi) \leq \sup_{\Omega_{\delta}} ( w^{\veps} + \eta \psi  - (w + \eta \psi))= \sup_{\Omega_{\delta}} ( w^{\veps} - w) \to 0,
		$$
		by the uniform convergence \eqref{conv1}. This shows that
		\begin{equation}\label{conv2}
		\sup_{\Omega_{\delta}} w^{\veps}_{\eta}  \searrow \sup_{\Omega_{\delta}} ( w + \eta \psi),\quad \text{$\veps \searrow 0$}.
		\end{equation}
		The same argument applied on $\partial \Omega_{\delta}$ gives
		\begin{equation}\label{conv3}
		\sup_{\partial \Omega_{\delta}} w^{\veps}_{\eta}  \searrow \sup_{\partial \Omega_{\delta}} ( w + \eta \psi),\quad \text{$\veps \searrow 0$}.
		\end{equation}
		Taking the limits in \eqref{conv2} and \eqref{conv3} in the estimate \eqref{Alex8} yields: for each $\eta > 0$ and with $\delta = \delta(\eta) > 0$ one has
		\begin{equation}\label{Alex9}
		\sup_{\Omega_{\delta}} w_{\eta} \leq \sup_{\partial \Omega_{\delta}} w_{\eta} + C,
		\end{equation}
		with $C$ defined in \eqref{Alex8}.
		
		Next we show how to let $\delta \to 0^+$ in \eqref{Alex9}. Given $\alpha > 0$, set 
		$$
		U_{\alpha} = \{ x \in \overline{\Omega}: \ w_{\eta}(x) < \sup_{\partial \Omega} w_{\eta} + C + \alpha\}$$
		and notice that $U_{\alpha}$ is an open neighborhood of $\partial \Omega$ in the topological space $\overline{\Omega}$. It is open because $U_{\alpha}$ is a strict sublevel set of $w_{\eta} \in \USC(\overline{\Omega})$ and contains $\partial \Omega$ because $\alpha$ is positive. Hence, for each $\alpha > 0$ fixed, there exists a $\delta_0 = \delta_0(\alpha) > 0$ such that
		$$
		\partial \Omega_{\delta} \cup (\Omega \setminus \Omega_{\delta}) \subset U_{\alpha}, \quad \forall \, \delta \in (0, \delta_0(\alpha))
		$$
		and hence
		\begin{equation}\label{Alex10}
		w_{\eta}(x) < \sup_{\partial \Omega} w_{\eta} + C + \alpha, \quad \forall \, \delta \in (0, \delta_0), \forall \, x \in \partial \Omega_{\delta} \cup (\Omega \setminus \Omega_{\delta}).
		\end{equation}
		Now, since $\delta(\eta) \to 0^+$ for $\eta \to 0^+$, for any fixed $\delta_{\ast} \in (0, \delta_0(\alpha))$, one has $\delta(\eta) \leq \delta_{\ast} \in (0, \delta_0(\alpha))$ for all small $\eta > 0$. Hence for all small $\eta > 0$ from \eqref{Alex9} we can conclude
		\begin{equation}\label{Alex11}
		\sup_{\Omega_{\delta(\eta)}} w_{\eta} \leq \sup_{\partial \Omega_{\delta(\eta)}} w_{\eta} + C
		< \sup_{\partial \Omega} w_{\eta} + C + \alpha, \quad \forall \, \alpha > 0,
		\end{equation}
		By combining \eqref{Alex11} with \eqref{Alex10}, for all $\eta > 0$ sufficiently small, we have ($\Omega = (\Omega \setminus \Omega_{\delta(\eta)}) \cup \Omega_{\delta(\eta)}$ )
		\begin{equation}\label{Alex12}
		\sup_{\Omega} w_{\eta} \leq \sup_{\partial \Omega} w_{\eta} + C + \alpha, \quad \forall \, \alpha > 0,
		\end{equation}
		and sending $\alpha \to 0^+$ yields
		\begin{equation}\label{Alex13}
		\sup_{\Omega} w_{\eta} \leq \sup_{\partial \Omega} w_{\eta} + C.
		\end{equation}
		Finally, since $w_{\eta} = w + \eta \psi \searrow w$ as $\eta \searrow 0$, with uniform convergence on $\overline{\Omega}$, the same argument used in the limit for $ \veps \searrow 0$ gives
		\begin{equation}\label{Alex14}
		\sup_{\Omega} w \leq \sup_{\partial \Omega} + C, \quad  C := \frac{\mathrm{diam}(\Omega)}{|B_1|^{1/n} \mathfrak{g}(I)^{1/N}} ||f^{n/N}||_{L^1(\Omega)}^{1/n},
		\end{equation}
		which completes the proof.
	\end{proof}

As mentioned in the introduction,  the oscillation estimate \eqref{ABPDG_osc} can be improved by replacing the $L^n$-norm on $\Omega$ by the $L^n$-norm on a smaller subset when the solution $h$ of \eqref{DGE} is $C^{1,1}$.

	\begin{rem}[Oscillation estimate for $C^{1,1}$ solutions]\label{rem:ABP_C11} Suppose that $(\mathfrak{g}, \Gamma)$ is an $I$-central G\aa rding-Dirichlet pair with $\mathfrak{g}$ of degree $N$ on $\cS(n)$ and supoose that $f \in C(\R^n)$ with $f \geq 0$. For each $\Omega \subset \R^n$ open and bounded one has the following Alexandrov oscillation estimate for $h \in C^{1,1}(\Omega) \cap C(\overline{\Omega})$ an admissible subsolution of $\mathfrak{g}(D^2h) = f$ on $\Omega$:
	\begin{equation}\label{ABPDG_osc_C11}
		\osc_{\Omega} h \leq \osc_{\partial \Omega} h + \frac{\mathrm{diam}(\Omega)}{|B_1|^{1/n}\mathfrak{g}(I)^{1/N}} || f ^{1/N}||_{L^n(E^{+}(-h))},
		\end{equation}
where $E^{+}(-h) = \Gamma^{+}(-h) \cap {\rm Diff}^2(-h)$ is the set of flat upper contact points of $-h$ at which it is twice differentiable.

We describe briefly why this is so. First, one continues to have the validity of the maximum principle of Lemma \ref{lem:Gamma_MP}, which controls the sup of $h$ on $\Omega$ by its sup on the boundary. Second, for $-h$ locally semiconvex one has the Alexandrov estimate of Theorem \ref{thm:ABP_lqc}

		\begin{equation} \label{ABP_sqc}
\sup_{\Omega} (-h) \leq \sup_{\partial \Omega} (-h) + \frac{\mathrm{diam}(\Omega)}{|B_1|^{1/n}} \left( \int_{	E^{+}(-h)} |\det \,  D^2(-h)(x)| \, dx \right)^{1/n},
\end{equation}
where $ D^2(-h)(x) \leq 0$ on $	E^{+}(-h)$ and hence $|D^2 (-h)(x)| = D^2 h(x) \geq 0$. The estimate \eqref{ABPDG_osc_C11} then follows by applying the determinant majorization of  Theorem \ref{thm:DME}. Notice that no semiconvex regularization is needed and that is why the domain of the integral error term can be left as $E^{+}(-h) = E^{-}(h) \subset \Omega$.
\end{rem}
	
	\section{Examples of $I$-central G\aa rding-Dirichlet polynomial operators}\label{sec:examples}
	
	In this section, we first recall various constructions (presented in \cite{HL09} and \cite{HL24}, for example), which illustrate the immensity of the family of $I$-central G\aa rding-Dirichlet polynomial operators  to which the main Theorem \ref{thm:ABP_P1} applies. Then we will discuss various examples which form interesting building blocks for our results. In the following section, we will discuss an extension of the class of $I$-central polynomial operators (which are not necessarily G\aa rding-Dirichlet) for which our method applies.
	
	\subsection{Building $I$-central G\aa rding-Dirichlet operators}
	
	The most classical example of a G\aa rding polynomial on $\cS(n)$ is the determinant, which is also an $I$-central G\aa rding-Dirichlet polynomial in the sense of Definition \ref{defn:ICGD_poly}.
	
	\begin{exe}[Monge-Amp\`ere operator]\label{exe:MA}
		The polynomial $\mathfrak{g}(A):= {\rm det}(A)$ for $A \in \cS(n)$ is homogeneous of degree $n$ and is $I$-hyperbolic since, for each $t \in \R$,
		$$
		\mathfrak{g}(A + tI) = {\rm det}(A + tI) = \prod_{k=1}(t + \lambda_k(A)),
		$$
		so that the G\aa rding $I$-eigenvalues are the standard real eigenvalues of $A \in \cS(n)$. The (open) G\aa rding cone is $\Gamma = \cP = \{A \in \cS(n):  A > 0 \}$ so that $\mathfrak{g}$ is G\aa rding-Dirichlet. Finally, a simple calculation shows that $D_I \mathfrak{g} = I$, so that $\mathfrak{g}$ is $I$-central.
		
		The determinant majorization estimate of Theorem \ref{thm:DME} is an equality
		\begin{equation}\label{DME:MA}
		{\rm det}(A)^{1/n} \geq {\rm det}(I)^{1/n} {\rm det}(A)^{1/n}  = {\rm det}(A)^{1/n}, \quad \forall \, A \geq 0 \ \text{in} \ \cS(n).
		\end{equation}
		The Alexandrov oscillation estimate
		\begin{equation}\label{ABPDG_osc:MA}
		\osc_{\Omega} h \leq \osc_{\partial \Omega} h + \frac{\mathrm{diam}(\Omega)}{|B_1|^{1/n}} || f ||_{L^1(\Omega)},
		\end{equation}
		holds for $\cP$-admissible viscosity solutions $h \in C(\overline{\Omega})$ to ${\rm det}(D^2h) = f$ on $\Omega$ with $f \in C(\Omega)$.
	\end{exe}
	
	Notice that we learn nothing new from the determinant majorization in this real case.
	% where there is equality in \eqref{DME:MA} without knowing that $\mathfrak{g} = {\rm det}$ is an $I$-central G\aa rding-Dirichlet polynomial. However, the example gains additional importance in light of the following constructions, which 
	Other basic examples arise by taking $\mathfrak{g}(A)$ to be either ${\rm det}(A_{\Cn})$ or ${\rm det}(A_{\Hn})$ where $A_{\Cn}$ is the complex hermitian part of $A:=D^2u$ and $A_{\Hn}$ is the quarternionic hermitian part of $A:=D^2u$ on complex or quaternionic manifolds, respectively.
	
	Notice that for the next constructions, as given in Proposition 3.1 of \cite{HL24}, construction (1) lowers the degree of while construction (2) raises the degree.
	
	\begin{rem}[Building $I$-central  G\aa rding-Dirichlet polynomials]\label{rem:constructions} On $\cS(n)$ one has the following constructions.
		\begin{enumerate}
			\item If $\mathfrak{g}$ is a $I$-central G\aa rding-Dirichlet polynomial of degree $N$, then the radial derivatives (also called derivative relaxations by Renegar \cite{R06})
			\begin{equation}\label{relax}
			\mathfrak{g}^{(\ell)}(A):= \left. \frac{d^{\ell}}{dt^{\ell}} \right\vert_{t = 0} \mathfrak{g}(A + tI), \quad \ell \in {1, \ldots N-1}
			\end{equation}
			are $I$-central G\aa rding-Dirichlet polynomials of degree $N - \ell$. The associated G\aa rding cones are nested
			$$
			\Gamma_{\mathfrak{g}} \subset 	\Gamma_{\mathfrak{g}^{(1)}} \cdots \Gamma_{\mathfrak{g}^{(N -1)}} = \Int \, \mathcal{H},
			$$
			where $\mathcal{H}$ is the Laplacian half-space. 
			\item The product $\mathfrak{gh}$ of two $I$-central G\aa rding-Dirichlet polynomials $\mathfrak{g}$ and $\mathfrak{h}$ on $\cS(n)$ is an $I$-central  G\aa rding-Dirichlet polynomial. The degree of  $\mathfrak{gh}$ is the sum of the degrees of $\mathfrak{g}$ and $\mathfrak{h}$ and the G\aa rding cones satisfy $\Gamma^{\mathfrak{gh}} = \Gamma^{\mathfrak{g}} \cap \Gamma^{\mathfrak{h}}$.
			
		\end{enumerate}
	\end{rem}
	See also Proposition 3.1 (3) for an additional construction of G\aa rding-Dirichet polynomials obtained by composing $\mathfrak{g}$ with a linear map $L \in {\rm End}(\cS(n))$, which preserves the degree of $\mathfrak{g}$. If, in 
	%\note{say something about how this is useful to reduce to diagonal matrices, which might be needed in Example \ref{exe:pfS}} 
	addition, $L(I) = kI$ for some $k > 0$ and ${\rm tr} \, A = 0 \ \Rightarrow \ {\rm tr} \, (L(A)) = 0$, then $\mathfrak{g}_L(A):= \mathfrak{g}(L(A))$ is $I$-central if $\mathfrak{g}$ is.
	
	\subsection{Some representative examples}
	
	We begin by noting that the radial derivatives in Remark \ref{rem:constructions} (1) applied to $\mathfrak{g} = {\rm det}$ yield the following interesting family of examples.
	
	\begin{exe}[$k$-Hessian operators]\label{exe:kH} Consider the elementary symmetric polynomials on $\R^n$
		\begin{equation}\label{KESP}
		\sigma_k(\lambda_1, \ldots , \lambda_n) := \sum_{1 \leq i_1 < \cdots < i_k \leq n} \left( \lambda_{i_1} \cdots \lambda_{i_k} \right), \quad k = 1, \ldots n,
		\end{equation}
		which define the homogeneous degree $k$ polynomials on $\cS(n)$ for each $k = 1, \ldots n$ by
		\begin{equation}\label{KHop}
		\mathfrak{g}_k(A) := \sigma_k(\lambda_1(A), \ldots , \lambda_n(A)) := \sum_{1 \leq i_1 < \cdots < i_k \leq n} \left( \lambda_{i_1}(A) \cdots  \lambda_{i_1}(A) \right), \quad A \in \cS(n), 
		\end{equation}
		where $\lambda_i(A)$ are the usual eigenvalues of $A$. 
		%Of course, $\mathfrak{g}_n = {\rm det} :=\mathfrak{g}$ defines the Monge-Amp\`ere operator, which is an $I$-central G\aa rding-Dirichlet polynomial. Hence its radial derivatives are also $I$-central G\aa rding-Dirichlet polynomials. 
		An easy calculation gives
		$$
		\mathfrak{g}^{(k)}(A):= \left. \frac{d^{k}}{dt^{k}} \right\vert_{t = 0} \mathfrak{g}(A + tI) = \mathfrak{g}_k(A), \quad k \in {1, \ldots n-1},
		$$
		proving that radial derivatives radial derivatives of ${\rm det}$ define the $k$-Hessian operators $\mathfrak{g}_k$ for $k = 1, \ldots n -1$.
		
		The associated G\aa rding cones are given by
		$$
		\Gamma_k := \{ A \in \cS(n): \sigma_j(\lambda_1(A), \ldots, \lambda_n(A)) > 0, \ j = 1, \ldots k \},
		$$
		whose subharmonics (as studied in \cite{BP21}) are sometimes called {\em $k$-convex functions} (beginning with \cite{TW97}; see also \cite{TW99}). Since $\mathfrak{g}_k(I) = \binom{n}{k}$, the determinant majorization estimate of Theorem \ref{thm:DME} yields (for each fixed $k = 1, \ldots, n$)
		\begin{equation}\label{DME:KH}
		\mathfrak{g}_k(A)^{1/k} \geq {\binom{n}{k}}^{1/k} {\rm det}(A)^{1/n}, \quad \forall \, A \geq 0 \ \text{in} \ \cS(n),
		\end{equation}
		which gives rise to the Alexandrov oscillation estimate
		\begin{equation}\label{ABPDG_osc:KH}
		\osc_{\Omega} h \leq \osc_{\partial \Omega} h + \frac{\mathrm{diam}(\Omega)}{\binom{n}{k}^{1/k} |B_1|^{1/n}} || f^{1/k} ||_{L^n(\Omega)}
		\end{equation}
		for $\overline{\Gamma}_k$-admissible viscosity solutions $h \in C(\overline{\Omega})$ to the inhomogeneous $k$-Hessian equation $\mathfrak{g}_k(D^2h) = f$ on $\Omega$ with $f \in C(\Omega)$.
	\end{exe}
	
	It is worth commenting on the determinant majorization inequality \eqref{DME:KH} of the family of examples in Example \ref{exe:kH}.
	
	\begin{rem}\label{rem:MacLaurin}[On determinant majorization by $k$-hessians] The family of determinant majorization inequalities \eqref{DME:KH} includes the two  ``extreme cases'' of $k = n$ and $k = 1$. When $k = n$, we have $\mathfrak{g}_n = {\rm det}$ and \eqref{DME:KH} is the equality \eqref{DME:MA} of Example \ref{exe:MA}. When $k = 1$, we have $\mathfrak{g}_1 = {\rm tr}$ and \eqref{DME:KH} is the matrix form of the {\em arithmetic-geometric mean inequality}
		\begin{equation}\label{AM-GM}
		{\rm tr}(A) \geq  {\rm det}(A)^{1/n}, \quad \forall  \, A \geq 0 \ \text{in} \ \cS(n).
		\end{equation}
		This inequality plays a key role in proving Alexandrov estimates for linear elliptic operators with variable coefficients (see, for example Section 2.5 of \cite{HanLin11}).
		
		Moreover, for the determinant majorization of \eqref{DME:KH} is actually a special case (with $\ell = n$) of the well-known {\em  Maclaurin's inequality}: for $1 \leq k \leq \ell$, one has
		\begin{equation}\label{MI}
		\frac{\mathfrak{g}_k(A)^{1/k}}{{\binom{n}{k}}^{1/k}} \geq \frac{\mathfrak{g}_{\ell}(A)^{1/{\ell}}}{{\binom{n}{\ell}}^{1/{\ell}}}, \quad \forall \, A \geq 0 \ \text{in} \ \cS(n).
		\end{equation}
		The reader might like to consult the recent paper of Tao \cite{T25}, which gives a variant of \eqref{MI} for the elementary symmetric polynomials \eqref{KESP} without requiring that the (eigenvalues) $\lambda_i$ be non negative.
	\end{rem}
	
	\begin{exe}[$p$-fold sum operators]\label{exe:pFS} Consider the following symmetric polynomials on $\R^n$
		\begin{equation}\label{pfS}
		\tau_k(\lambda_1, \ldots \lambda_n) := \prod_{1 \leq i_1 < \cdots < i_p \leq n} \left( \lambda_{i_1} + \cdots + \lambda_{i_p} \right), \quad p = 1, \ldots n,
		\end{equation}
		which define the homogeneous degree polynomials of degree $N = \binom{n}{p}$ on $\cS(n)$ for each $p = 1, \ldots n$ by
		\begin{equation}\label{pfSop}
		\mathfrak{g}_p(A) := \tau_p(\lambda_1(A), \ldots , \lambda_n(A)) := \prod_{1 \leq i_1 < \cdots < i_k \leq n} \left( \lambda_{i_1}(A) + \cdots + \lambda_{i_p}(A) \right), \quad A \in \cS(n), 
		\end{equation}
		where $\lambda_i(A)$ are the usual eigenvalues of $A$. 
		
		This operator was introduced and studied in \cite{HL09}. It is geometrically significant because its subharmonics are precisely the upper semicontinuous functions that restrict to all affine $p$-planes as classical (Laplacian) subharmonics.
		
		For each $p = 1, \ldots, n$, $\mathfrak{g}_p$ defines an $I$-central G\aa rding-Dirichlet operator as noted in Example 3.3 of \cite{HL23}. For the benefit of the reader, we review why this is so. One easily checks that they are G\aa rding, where the $N = \binom{n}{p}$ G\aa rding $I$-eigenvalues are the factors in the product of \eqref{pfSop}. They are aslo G\aa rding-Dirichlet where the (closed) G\aa rding cones (after ordering the eigenvalues $\lambda_1(A) \leq \cdots \leq \lambda_n(A)$) are given by
		\begin{equation}\label{pfS_cones1}
		\overline{\Gamma}_p := \{ A \in \cS(n): \lambda_1(A) + \cdots \lambda_p(A) \geq 0\}
		\end{equation}
		and are nested
		\begin{equation}\label{pfS_cones2}
		\cP = \overline{\Gamma}_1 \subset \cdots \subset \overline{\Gamma}_p \subset \cdots \subset \overline{\Gamma}_n = \mathcal{H},
		\end{equation}
		where $\mathcal{H}$ is again the (Laplacian) half-space and $\cP$ is the convexity constraint set. In fact, this family of examples gives another interpolation between the Laplacian and the Monge-Ampere operator where here $\mathfrak{g}_1 = {\rm det}$ and $\mathfrak{g}_n = {\rm tr}$.
		
		To see that the operators are $I$-central, it is enough to note that $\mathfrak{g}_p$ is $\mathcal{O}(n)$ invariant (see Proposition E.5 of \cite{HL23}) where the constant $k$ of Definition \ref{defn:ICGD_poly} (3) is $k = \mathfrak{g}_p(I) N/n$ with $N = \binom{n}{p}$ (see Theorem E.8 of \cite{HL23}, where the normalization $\mathfrak{g}(I) = 1$ is assumed.)
		
		Since $\mathfrak{g}_p(I) = pN := p \binom{n}{p}$, the determinant majorization estimate of Theorem \ref{thm:DME} yields (for each fixed $p = 1, \ldots, n$)
		\begin{equation}\label{DME:pfS}
		\mathfrak{g}_p(A)^{1/N} \geq (pN)^{1/N} {\rm det}(A)^{1/n}, \quad \forall \, A \geq 0 \ \text{in} \ \cS(n), \quad \text{where} \ N = \binom{n}{p},
		\end{equation}
		which gives rise to the Alexandrov oscillation estimate
		\begin{equation}\label{ABPDG_osc:pfS}
		\osc_{\Omega} h \leq \osc_{\partial \Omega} h + \frac{\mathrm{diam}(\Omega)}{(pN)^{1/N} |B_1|^{1/n}} || f^{1/N} ||_{L^n(\Omega)}, \quad \text{where} \ N = \binom{n}{p},
		\end{equation}
		for $\overline{\Gamma}_p$-admissible viscosity solutions $h \in C(\overline{\Omega})$ to the inhomogeneous $p$-fold sum equation $\mathfrak{g}_p(D^2h) = f$ on $\Omega$ with $f \in C(\Omega)$. 
		
		It should be noted that the determinant majorization estimate \eqref{DME:pfS} in this case was first proven by Dinew in \cite{Di23}.
	\end{exe}
	
%	\noindent{\bf Working note:} Add additional \note{Maybe we can take this out. What do you think, Reese?} examples like
%	\begin{enumerate}
%		\item Lagrangian Monge-Amp\`ere operator (on $\R^{2n}$)?
%	\end{enumerate}
	%If so, what are the G\aa rding cones, $\mathfrak{g}(I)$ and resulting determinant majorization/oscillation estimates?
	
	\section{Extensions to non-G\aa rding polynomial operators}\label{sec:NG_Ops}
	
	Up to now, we have considered inhomogeneous equations $\mathfrak{g}(D^2 u) = f$ where $\mathfrak{g}$ is an $I$-central G\aa rding-Dirichlet polynomial. This class of polynomial operators satisfies two crucial ingredients for our approach; namely, a sufficient condition for the validity of the determinant majorization estimate of Theorem \ref{thm:DME} and a well-developed operator theory as presented in Remark \ref{rem:DGOps1} and Remark \ref{rem:DGOps2}. We can, and will, split these two elements into two separate questions using considerations which were first presented in \cite{HL24}. We will consider inhomogeneous equations of the form
	\begin{equation}\label{inhom}
	G(D^2_x u) = f(x), \quad x \in \Omega \subset \subset \R^n
	\end{equation}
	where $G$ is a homogeneous polynomial of degree $N$ on $\cS(n)$ which is $I$-central in the sense of Definition \ref{defn:ICGD_poly} (3), but not necessarily G\aa rding in the sense of Definition \ref{defn:ICGD_poly} (1). 
	
	The main points are two. First, if $G$ is an $I$-central polynomial and satisfies an additional {\bf {\em coefficient condition}} (see condition (ii) of Theorem \ref{thm:DME:NG}), determinant  majorization continues to hold. 
	%\note{Rewritten, since we will consider only tame operators} 
	Second, additional assumptions on $G$ can ensure that $F(x,A):= G(A) - f(x)$ has an operator theory suitable for our approach. More precisely, we can ensure that $G$ continues to be degenerate elliptic (as defined in Remark \ref{rem:DGOps1}(b)) %and topologically tame (defined in \ref{rem:DGOps2}(c))
	and tame (as defined in (as defined in Remark \ref{rem:DGOps2}(a)). The tameness of $G$ can be then used to prove the fiberegularity of the pair $(F, \cG)$ as in Remark \ref{rem:correspondence}. This fiberegularity then gives the needed correspondence principle and semiconvex approximation of subharmonics. These ingredients are automatic for G\aa rding-Dirichlet polynomials $G = \mathfrak{g}$ and $f$ continuous and non-negative.
	
	We begin with the following alternative to Theorem \ref{thm:DME}, which is Theorem 2.3 of  \cite{HL24}. In what follows, we denote by $\mathcal{O}(n) \subset \cS(n)$ the set of orthogonal matrices and 
	$$
	\mathcal{D}(n) := \{ X \in \cS(n): \ X = {\rm diag}(x_1, \ldots, x_n)\}
	$$
	the set of diagonal matrices. 
	
	\begin{thm}[Determinant majorization for non-G\aa rding $I$-central polynomials]\label{thm:DME:NG} Suppose that $G$ is a homogeneous polynomial of degree $N$ on $\cS(n)$ which satisfies the follwoing two conditions:
		\begin{itemize}
			\item[(i)] $G$ is $I$-central in the sense of Definition \ref{defn:ICGD_poly} (3); that is,
			\begin{equation}\label{IC_poly}
			\exists \, k> 0 \ \text{such that} \ D_I G = k I;
			\end{equation}
			\item[(ii)] For every $\tau \in \mathcal{O}(n)$ the polynomial $\mathfrak{p}_{\tau}$ on $\R^n$ defined by
			\begin{equation}\label{pO}
			\mathfrak{p}_{\tau}(x_1, \ldots, x_n) := G(\tau X \tau^t), \quad X = {\rm diag}(x_1, \ldots, x_n) \in \mathcal{D}(n)
			\end{equation} 
			has non negative coefficients.\\
			This will be referred to as the \underline{coefficient condition} for $G$.
		\end{itemize}
		Then one has the determinant majorization
		\begin{equation}\label{DME:NG}
		G(A)^{1/N} \geq 	G(I)^{1/N} \left( {\rm det} A \right)^{1/n}, \quad \forall \, A \geq 0 \ \text{in} \ \cS(n).  
		\end{equation}
	\end{thm}
	
	For polynomial operators $G$ satisfying the hypotheses of Theorem \ref{thm:DME:NG}, the validity of the Alexandrov estimates (like those in Theorem \ref{thm:ABP_P1}) depends on establishing that there exists a subequation $\G = \R \times \R^n \times \mathcal{A} \subset \R \times \R^n \times \cS(n)$ such that $(F, \G)$ is a compatible and fiberegular $\cM$-monotone operator-subequation pair where the operator $F$ is defined by $F(x,A) := G(A) - f(x)$. This will guarantee the validity of a correspondence principle analogous to Theorem \ref{thm:CPP} for $\G$-admissible solutions to \eqref{inhom} as well as the semiconvex approximation result of Lemma \ref{lem:QCA} $\wt{\F}$-subharmonics with $\wt{\F}$ defined by
	\begin{equation*}\label{Fdualsub_NGG}
	\wt{\F}_x := \{(r,p,A) \in  \cJ^2: \ \mbox{$-A \not\in \mathcal{A}$ or  [$-A \in \mathcal{A}$ and $G(-A) - f(x) \leq 0$]}\}, \quad x \in \Omega.
	\end{equation*}  . The following example was introduced in \cite{HL24} and is 
	%\note{A good name for the example?} 
	illustrative of what what can be done 
	
	\begin{exe}\label{exe:NGMA} Consider the homogeneous polynomial operator of degree $2 + n$
		\begin{equation}\label{PMA1}
		G(A):= ||A||^2 {\rm det}(A), \quad A \in \cS(n).
		\end{equation} 
	\end{exe}
	
	\begin{rem}[Failure to be G\aa rding]\label{rem:NG}
		The polynomial defined by \eqref{PMA1} fails to be G\aa rding; that is, it fails to be $I$-hyperbolic since, for each fixed  $A \in \cS(n)$ and $t \in \R$,
		\begin{equation}\label{PMA2}
		G(A + tI) = || A + tI||^2 +  {\rm det}(A + tI) = \left[ ||A||^2 + 2 ({\rm tr}\, A)t + t^2 \right] \prod_{k=1}^n \left(t + \lambda_k(A) \right)
		\end{equation}
		will have real roots in the first factor only when $| {\rm tr} \, A| \geq ||A||$. However, it is $I$-central where
		$$
		D_I G = (2 + n)I
		$$
		and does satisfy the coefficient condition (ii) of Theorem \ref{thm:DME:NG} since for each $\tau\in \mathcal{O}(n)$
		$$
		\mathfrak{p}_{\tau}(x_1, \ldots, x_n) := G(\tau X \tau^t) = |x|^2 \prod_{k=1}^n x_k, \quad \forall \, X = {\rm diag}(x_1, \ldots, x_n) \in \mathcal{D}(n),
		$$
		has non-negative coefficients.
	\end{rem}
	
	%\begin{rem}[Alexandrov estimates for sufficiently regular functions]\label{rem:AE_NG_reg}
	%maybe come back to this
	%	\end{rem}
	
	In order to apply the determinant majorization estimate of Theorem \ref{thm:DME:NG} to (admissible) viscosity solutions of the inhomogeneous equation \eqref{thm:DME:NG} with $G$ defined by \eqref{PMA1}, we will need a good operator theory for the operator $F(x,A):= G(A) - f(x)$ with respect to a natural admissibility constraint $\cG$. 
	
	We first discuss the operator theory associated to $G$. The polynomial operator $G$ defined by \eqref{PMA1} has a natural associated subequation $\G:= \R \times \R^n \times \cP$, which renders $G$ degenerate elliptic when using $\cG$ as an admissibility constraint. It is straightforward to check that $(G,\G)$ is an $\cM$-monotone operator subequation pair with monotonicity cone subequation $\cM = \cM(\cP)$ where $\cM:= \R \times \R^n \times \cP = \G$. Moreover, the pair is clearly compatible since for each
	$$
	\partial \G = \{ (r,p,A) \in \G: G(A) = 0\}.
	$$
	Moreover, $G$ is tame on $\cG$ which is the last key ingredient.
	
	\begin{rem}[Tameness of $G$]\label{rem:tameness} The operator defined by $G(A):=||A||^2 {\rm det}(A)$ satisfies
		\begin{equation}\label{G_tame}
		G(A + \eta I) - G(A) \geq n \eta^{n+2}, \quad \forall \, A \in \cP, \forall \, \eta \geq 0;
		\end{equation} 
	that is, $G$ is tame on $\cP \subset \cS(n)$. To see that \eqref{G_tame} holds, notice first that
	\begin{equation*}\label{Gtame1}
	||A + \eta I||^2 = \langle A + \eta I, A + \eta I \rangle = ||A||^2 + 2 \eta \, {\rm tr} A + n \eta^2,
	\end{equation*}
	so that 
		\begin{equation}\label{Gtame2}
G(A + \eta I) - G(A) = 	\left( ||A||^2 + 2 \eta \, {\rm tr} A + n \eta^2 \right) {\rm det}(A + \eta I) - ||A||^2 {\rm det} A. 
	\end{equation}	
Next, use the well-known expansion in terms of the elementary symmetric functions $\sigma_k$ (where $\sigma_1 = {\rm tr}$ and $\sigma_n = {\rm det}$):
\begin{equation}\label{Gtame3}
{\rm det}(A + \eta I) = \sum_{k = 1}^n \eta^{n -k} \sigma_k(A) = {\rm det} A   +  \eta \sigma_{n-1}(A) + \ldots + \eta^n.
\end{equation}
Since 
\begin{equation}\label{Gtame4}
\sigma_k(A) \geq 0, \quad \forall \, k = 1, \ldots n, \forall \, A \in \cP,
\end{equation}
combining \eqref{Gtame2}-\eqref{Gtame4}, one obtains: $\forall \, A \in \cP, \forall \, \eta \geq 0$,
$$
G(A + \eta I) - G(A) \geq 	\left( ||A||^2 + n \eta^2 \right) \left(  {\rm det} A   +  \eta  \eta^n \right) - ||A||^2 {\rm det} A = n \eta^{n+2}, 
$$
which is \eqref{G_tame}.	
	\end{rem} 
It is now straightforward to show the following result for the model operator of Example \ref{exe:NGMA}.

	\begin{thm}\label{thm:ANG}
		Suppose that $\Omega \subset \R^n$ is open and bounded, $f \in C(\Omega)$ is non-negative and $G$ is the polynomial on $\cS(n)$ defined by $G(A):= ||A||^2 \, {\rm det}(A)$. Denoting by $\cG:= \R \times \R^n \times \cP$,  the following statements hold.
		\begin{itemize}
			\item[(a)] The constraint sets defined fiberwise by
			\begin{equation}\label{Fsub_NG}
			\F_x := \{(r,p,A) \in  \cJ^2:  A \in \G \quad \text{and} \quad G(A) - f(x) \geq 0\}, \quad x \in \Omega
			\end{equation}
			and 
			\begin{equation}\label{Fdualsub_NG}
			\wt{\F}_x :=\{(r,p,A) \in  \cJ^2: -A \not\in \G  \ \text{or} \  [-A \in \G \ \text{and} \ G(-A) - f(x) \leq 0] \}, \quad x \in \Omega.
			\end{equation}
			are $\cM$-monotone fiberegular subequations for the monotonicity cone subequation $\cM = \cM(\cP) = \R \times \R^n \times \cP = \cG$.
			\item[(b)] The correspondence principle (as stated in Theorem \ref{thm:CPP}) holds for solutions of the inhomogeneous equation 
			\begin{equation}\label{DGE_NG}
			G(D^2h) = f \quad \mbox{with the admissibility constraint $D^2h \in \cG$}.
			\end{equation}
		\end{itemize}
		Combining parts (a) and (b) with the determinant majorization of Theorem \ref{thm:DME:NG}, one has the following Alexandrov estimates.
		\begin{itemize}
			\item[(c)] If $u \in \USC(\overline{\Omega})$ is an admissible subsolution of \eqref{DGE_NG} on $\Omega$ (equivalently, $u$ is $\F$-subharmonic on $\Omega$ for $\F$ defined by \eqref{Fsub_NG}) then
			\begin{equation}\label{ABPDG_sub_NG}
			\sup_{\Omega} u \leq \sup_{\partial \Omega} u;
			\end{equation}
			that is, the maximum principle ${\rm (MP)}$ holds.
			\item[(d)] If  $v \in \LSC(\overline{\Omega})$ is an admissible supersolution of \eqref{DGE_NG} on $\Omega$ (equivalently, $v$ is $\F$-superharmonic on $\Omega$) then
			\begin{equation}\label{ABPDG_super_NG}
			\inf_{\Omega} v \geq \inf_{\partial \Omega} v - \frac{\mathrm{diam}(\Omega)}{|B_1|^{1/n}} || f ^{1/(2+n)}||_{L^n(\Omega)}.
			\end{equation}
			Equivalently, $w:=-v \in \USC(\Omega)$ is $\wt{\F}$-subarmonic on $\Omega$ with $\wt{\F}$ defined by \eqref{Fdualsub_NG} and satisfies 
			\begin{equation}\label{ABPDG_dual_NG}
			\sup_{\Omega} w \leq \sup_{\partial \Omega} w + \frac{\mathrm{diam}(\Omega)}{|B_1|^{1/n}} || f ^{1/(2+n)}||_{L^n(\Omega)}.
			\end{equation}
			\item[(e)] If  $h \in C(\overline{\Omega})$ is an admissible solution of \eqref{DGE_NG} in $\Omega$ (equivalently, $h$ is $\F$-harmonic on $\Omega$) then
			\begin{equation}\label{ABPDG_osc_NG}
			\osc_{\Omega} h \leq \osc_{\partial \Omega} h + \frac{\mathrm{diam}(\Omega)}{|B_1|^{1/n}} || f ^{1/(2+n)}||_{L^n(\Omega)}.
			\end{equation}
		\end{itemize}
	\end{thm}
	
	\begin{proof}
		For the claims in part (a), the constraint sets are clearly $\cM$ monotone for $\cM = \R \times \R^n \times \cP$ and hence satisfy two of the three subequation axioms; namely positivity (P) and negativity (N). For the remaining axiom of topological stability (T), it suffices to show (see, for example, Theorem 7.11 of \cite{CPR23}) that $(F, \cG)$ is a fiberegular operator subequation pair,  where $F$ is the operator $F(x,A):= G(A) - f(x)$.
		
		As recalled in Remark \ref{rem:correspondence} (2), since $\cG$ is fiberegular (it has constant fibers) and the pair is $\cM$-monotone, the fiberegularity of the pair $(F, \cG)$ follows from the uniform continuity of $f$ on relatively compact subsets $\Omega'$ of $\Omega$ and the tameness of $G$ (which holds by Remark \ref{rem:tameness}).

		The correspondence principle of part (b) follows from the fact that $(F, \G)$ is a compatible proper elliptic operator-subequation pair and the constraint set $\F$ defined by \eqref{Fsub_NG} is a subequation by part (a)  (see Theorem 7.4 of \cite{CPR23}).
		
		The Alexandrov estimates of parts (c), (d) and (e) are proven exactly as in Theorem \ref{thm:ABP_P1} where the determinant majorization of Theorem \ref{thm:DME:NG} replaces that of Theorem \ref{thm:DME}. Notice that for maximum principle of part (c) continues to hols as the admissible subsolutions are convex. Notice also that parts (a) and (b) ensure that the semiconvex approximation for dual subharmonics of Lemma \ref{lem:QCA} extends to this example. Finally, notice that $G(I) = 1$ in this example, so that the form of the constant in the integral term of \eqref{ABPDG_dual_NG} and \eqref{ABPDG_osc_NG} simplifies.

	\end{proof}
	\begin{rem} As noted in \cite{HL24}, the above 
		%\note{Do we want to put in the case $\mathfrak{g(A) =  {\rm tr}} A$ now, or keep it until we understand it better. I vote for the second option.} 
		example should generalize to $G(A) := ||A||^2 \mathfrak{g}(A)$ for any $I$-central G\aa rding-Dirichlet polynomial of degree $N$ on $\cS(n)$. The failure, in general, to be an $I$-hyperbolic polynomial again follows from \eqref{PMA2}, where one merely replaces the standard eigenvalues $\lambda_k(A) \in \R$ with the G\aa arding $I$-eigenvalues $\lambda_k^{\mathfrak{g}}(A) \in \R$. The other ingredients which lead to Alexandov estimates proceeds as in Theorem \ref{thm:ANG}.
	\end{rem}
	
	The geometric relevance of the class of operators in Example \ref{exe:NGMA} remains to be investigated. Many variants of Example \ref{exe:NGMA} are obviously possible, also in light of Remark \ref{rem:constructions}. We will not pursue this here.

	\begin{appendix}
		
		\section{Pure second order inhomogeneous equations}\label{sec:App_A}

We give a brief summary of the definitions and conditions necessary to make sense of viscosity solutions for any inhomogeneous equation
	\begin{equation}\label{inhom_A}
G(D^2_x u) = f(x), \quad x \in \Omega \subset \R^n
\end{equation}
associated to a constant coefficient pure second order operator $G$ and an inhomogeneous term $f$. Before proceeding, a simplifying remark is in order.

\begin{rem}[Reduction]\label{rem:AppA} To simplify notation, we will make use of the fact that the equation \eqref{inhom_A} is {\em reduced} in the sense that it only depends on the jet variables $(x,A) \in \Omega \times \cS(n)$ and not on the other jet variables $(r,p) \in \R \times \R^n$ that track dependence on the values of $u$ and on the values of the gradient of $u$ for solutions $u$ of \eqref{inhom_A} (see Chapter 10 of \cite{CHLP20} for an organic discussion of reducing jet variables). We could have done this throughout the paper, but in order to easily apply certain results valid in the general case (dependence on all jet variables), we maintained track of all jet variables up to now.
	\end{rem}

First, one must choose a domain $\cA \subset \cS(n)$ (called the {\em admissibility set}) for $G$. We impose four conditions on the pair $(G,\cA)$.
\begin{itemize}
	\item[(1)] $\cA$ is a {\em subequation}; that is, $\cA$ is a closed subset satisfying {\em positivity}
	\begin{equation*}\label{P}\tag{P}
	A + P \in \cA, \quad \forall \, A \in \cA \ \text{and} \ P \geq 0.
	\end{equation*}
	\end{itemize}
Also,
\begin{itemize}
	\item[(2)] the operator $G \in C(\cA)$ must be {\em degenerate elliptic}; that is, 
	\begin{equation*}\label{DE}\tag{DE}
	G(A + P) \geq G(A), \quad \forall \, A \in \cA \ \text{and} \ P \geq 0.
	\end{equation*}
\end{itemize}
In this paper, we have considered the case where the domain $\cA$ of the operator $G$ is {\em constrained} to be a proper subset of $\cS(n)$, with one exception; namely the linear case $G(A) := \langle E, A \rangle$ for some fixed $E \geq 0$, for which the domain need not be constrained.
\begin{itemize}
	\item[(3)] In the constrained case, where $\cA \subsetneq \cS(n)$, assume that the operator-subequation pair $(G, \cA)$ is {\em compatible}; that is,
		\begin{equation*}\label{C}\tag{C}
	\partial \cA = \{ A \in \cS(n): \ G(A) = 0\} \quad \text{and} \quad \Int \, \cA = \{ A \in \cS(n): \ G(A) > 0\}.
	\end{equation*}
	\end{itemize}

Given an inhomogeneous term $f \in C(\Omega)$ with $f \geq 0$ on a domain $\Omega \subset \R^n$, conditions (1) and (2) on the pair $(G, \cA)$ imply that the constraint set
\begin{equation}\label{Inhom:subeqn}
\F := \{ (x,A) \in \Omega \times \cA: \ G(A) - f(x) \geq 0\}
\end{equation}
is a (pure second order) {\em subequation on $\Omega$} in the sense 
%\note{more precise, other citations} 
of \cite{HL11} and hence has a well-defined potential theory for its subhrmonics, superharmonics and harmonics.

Subharmonics for $\F$ are the same thing as subsolutions to $G(D^2_xu) = f(x)$. They can be defined via viscosity {\em upper test functions} (see Definition \ref{defn:UCJ}). A function $u \in \USC(\Omega)$ is said to be $\F$-subharmonic in $\Omega$ if for every $x \in \Omega$ and for each upper test function $\varphi$ for $u$ at $x$
\begin{equation}\label{Sub1}
A = D^2 \varphi(x) \in \cA \quad \text{and} \quad (x,A)  \in \F \quad (\text{equivalently} \ G(A) - f(x) \geq 0),
\end{equation}
where the equivalence above defines $u$ being an $\cA$-admissible viscosity subsolution of \eqref{inhom_A}.
This equivalence is one half of the {\em correspondence principle} in this situation.

Before completing the correspondence principle, a fourth condition is necessary on the operator $G$. This rules out rather silly pathologies.
\begin{itemize}
	\item[(4)] (Topological tameness) For each level $c \geq 0$, the level set $\{A \in \cA: \ G(A) = c \}$ has no interior.
	\end{itemize}
\begin{rem}[On topological tameness]\label{rem:TT} Notice the following.
\begin{itemize}
	\item[(a)] Without topological tameness, one certainly will not have uniqueness for the natural Dirichlet problem associated to the equation \eqref{inhom_A} (nor comparison principles which imply uniqueness). Far from it, pick $A$ in the interior of the $c$-level set of $G$ for some $c \geq 0$ and pick a smooth function $\varphi$ (possibly compactly supported). Then, for all $\veps$ sufficiently small, the smooth functions $\frac{1}{2} \langle Ax, x \rangle + \veps \varphi(x)$ all satisfy the inhomogeneous equation $G(D^2u) = c$.
	\item[(b)] Obviously, for polynomial operators $G$ (which are real analytic), this topological tameness requirement is satisfied.  
\end{itemize}
\end{rem}

Under these four assumptions, one has all that is needed to treat viscosity solutions of \eqref{inhom_A} by the potential theoretic method initiated in \cite{HL09}. In particular, one has the natural notion of a {\em dual subharmonic function} defined as those $v \in \USC(\Omega)$ which are subharmonic (in the sense above) for the dual (pure second order) subequation with fibers
$$
 \wt{F}_x := -  \Int \, (\F_x)^c \quad \text{(with interior relative to $\cS(n)$)},
$$
which is equivalent to $w = -v$ being an $\cA$-admissible supersoltion of \eqref{inhom_A}. This is the other half of the correspondence principle. Putting these two halves together,  one concludes that  {\em $\F$-harmonic functions on $\Omega$} are precisely the ($\cA$ constrained)  {\em solutions of $G(D^2u) = f$ on $\Omega$}. 

Various generalizations of this correspondence  principle are known. See Theorem 11.13 of \cite{CHLP20} for
$G$ not necessarily a pure second order operator, but which would require the inhomogeneous term $f$ to be constant. See Theorems 7.4 and 7.11 of \cite{CPR23} for a generalization that allows for nonconstant $f$.

		\section{Proof of Lemma \ref{lem:QCA}}\label{sec:App_B}
		
		For the convenience of the reader, we recall the necessary definitions and the statement of the result. For each $\veps > 0$ and each $\eta > 0$ define the perturbed sup-convolution approximations \eqref{approx1}
		\begin{equation}\label{approx1A}
		w_{\eta}^{\veps}(x) := w^{\veps}(x) +  \frac{\eta}{2} |x|^2 \quad \forall \, x \in \Omega,
		\end{equation}
		where 
		\begin{equation}\label{approx2A}
		w^{\veps}(x) := \sup_{y \in \Omega} \left( w(y) - \frac{1}{2 \veps} |y - x|^2 \right), \quad \forall \, x \in \Omega,
		\end{equation}
		is the $\veps$-sup-convolution of $w$ of \eqref{approx2}, which is $\frac{1}{\veps}$-semiconvex on $\Omega$ and decreases pointwise on $\Omega$ to $w$.
		
		\begin{lem}[Semiconvex approximation of dual subharmonics]\label{lem:QCAA} Suppose that $\wt{\F}$ is an $\cM(\cP)$-monotone and fiberegular subequation (with modulus of continuity $\delta = \delta(\eta)$) on $\Omega$ open and bounded.  Suppose that $w \in \wt{\F}(\Omega)$ is bounded, with $|w| \leq M$ on $\Omega$. Then for each $\eta > 0$ there exists $\veps_{\ast} = \veps_{\ast}(\delta(\eta), M) > 0$ such that the perturbed sup-convolutions defined by \eqref{approx1A} satisfy
			\begin{equation}\label{approxA}
			w_{\eta}^{\veps} \in \wt{\F}(\Omega_{\delta}), \quad \forall \, \veps \in (0, \veps_{\ast}],
			\end{equation}
			where $\Omega_{\delta}:= \{ x \in \Omega: \ {\rm dist}(x, \partial \Omega) > \delta\}$.
		\end{lem} 
		
		Before giving the proof, we briefly describe the main ideas. The proof involves two main steps. Step 1 concerns the so-called {\bf {\em uniform translation property}} that is satisfied by any fiberegular $\cM(\cP)$-monotone subequation. Step 1 is a special case of Theorem 3.3 of \cite{CPR23} adapted to present situation of fiberegular pure second order subequations (where the monotonicity cone subequation is $\cM = \cM(\cP)$) and stated for a dual subequation $\wt{\F}$ (which is also $\cM$-monotone and fiberegular with the same modulus of continuity $\delta$. For the convenience of the reader we present a simplified proof in this special case. Step 2 exploits the uniform translation property together with the boundedness of $w$ in order to represent the quadratically perturbed sup-convolutions $w_{\eta}^{\veps}$ of $w$ as the Perron function for a family of $\wt{\F}$-subharmonics which are locally bounded from above on $\Omega_{\delta}$.

		\begin{proof} Suppose that $\wt{\F}, \Omega$ and $w$ satisfy the hypotheses of the lemma.

			\noindent \underline{Step 1:} (The Uniform Translation Property) {\em For each $\eta > 0$ and for each $z \in \R^n$ with $|z| < \delta$ (where $\delta = \delta(\eta)$ is the modulus of continuity of $\wt{\F}$), the quadratically perturbed translates of $w$ defined by}
			\begin{equation}\label{w_z_eta}
			w_{z; \eta}(x) := w(x-z) + \frac{\eta}{2}|x|^2, \quad x \in \Omega_{\delta}
			\end{equation}
			{\em are $\wt{\F}$-subharmonic on $\Omega_{\delta}$}.
			
			To prove this claim, we make use of the {\em definitional comparison lemma} for $w = w_{z; \eta}$ on $ \Omega_{\delta}$ (see Lemma B.2 of \cite{CPR23}, which holds for \underline{any} subequation on any open set $X \subset \R^n$). In our context, it says that for $w \in \USC(X)$ one has two statements:
			\begin{enumerate}
				\item {\em if $w$ is $\wt{\F}$-subharmonic on $X$, then for every open subset $\Omega \subset \subset X$ and for every $v \in \USC(\overline{\Omega}) \cap C^2(\Omega)$ which is strictly $\F$-subharmonic on $\Omega$}
				\begin{equation}\label{DC1}
				w +v \leq 0 \ \text{{\em on}} \ \partial \Omega \ \ \Rightarrow \ \ w +v \leq 0 \ \text{{\em on}} \ \Omega;
				\end{equation}
				\item {\em if for each $x \in X$ there exists a small ball $B \subset X$ containing $x$ such that for each $v \in \USC(\overline{B}) \cap C^2(B)$ which is strictly $\F$-subharmonic on $B$}
				\begin{equation}\label{DC2}
				w +v \leq 0 \ \text{{\em on}} \ \partial B \ \ \Rightarrow \ \ w +v \leq 0 \ \text{{\em on}} \ B,
				\end{equation}
				{\em then  $w$ is $\wt{\F}$-subharmonic on $X$}.
			\end{enumerate}
			For each $x \in X := \Omega_{\delta}$ we choose $B := B_{\delta}(x) \subset \Omega$. In order to show that $w_{z; \eta}$ is $\wt{\F}$-subharmonic on $\Omega_{\delta}$, by part (2) of definitional comparison \eqref{DC2}, it suffices to show that for each $v \in \USC(\overline{B}) \cap C^2(B)$ which is strictly $\F$-subharmonic on $B$ we have
			\begin{equation}\label{DC3}
			w_{z; \eta} +v \leq 0 \ \text{{\em on}} \ \partial B \ \ \Rightarrow \ \ w_{z; \eta} +v \leq 0 \ \text{{\em on}} \ B.
			\end{equation}
			Notice that for each $x \in \Omega_{\delta}$ and for each $z \in B_{\delta}(0)$ we have the identity
			\begin{equation}\label{NTF1}
			\left(w_{z; \eta} + v \right)(x) := w(x-z) + \frac{\eta}{2}|x|^2 + v(x) = w(x-z) + \widehat{v}_{z;\eta}(x -z),
			\end{equation}
			if we define the translated quadratic perturbations of $v$ by
			\begin{equation}\label{NTF2}
			\widehat{v}_{z;\eta}(x):= v(x + z) + \frac{\eta}{2}|x|^2, \quad x \in B - \{z\}, |z| < \delta, \eta > 0.
			\end{equation}
			By hypothesis, $w$ is $\wt{F}$-subharmonic on $\Omega$ and by restriction it is  $\wt{F}$-subharmonic on $B - \{z\} \subset \Omega_{\delta} - \{z\} \subset \Omega$. It suffices to show that $\widehat{v}_{z;\eta}$, which belongs to $\USC(\overline{B - \{z\}}) \cap C^2(B)$, satisfies
			\begin{equation}\label{NTF3}
			\mbox{ $\widehat{v}_{z;\eta}$ is strictly $\F$-subharmonic on $B - \{z\}$.}
			\end{equation}
			Indeed, by applying part (1) of definitional comparison to the pair $(w, \widehat{v}_{z;\eta})$ on $B - \{z\}$, we have
			\begin{equation}\label{DC4}
			w + \widehat{v}_{z; \eta} \leq 0 \ \text{{\em on}} \ \partial(B - \{z\}) = \partial B - \{z\} \ \ \Rightarrow \ \ w + \widehat{v}_{z; \eta} \leq 0 \ \text{{\em on}} \ B - \{z\},
			\end{equation}
			which by the identity \eqref{NTF1} is equivalent to \eqref{DC3}.
			
			It remains 
			%\note{This is the main point and where fiberegulrity/tameness gets used.} 
			to prove the claim \eqref{NTF3} for $\widehat{v}_{z;\eta}$ defined by \eqref{NTF2}. Since $\widehat{v}_{z;\eta} \in C^2(B - \{z\})$, we need only that the 2-jets must satisfy
			\begin{equation}\label{NTF4}
			J^2_{x-z} \widehat{v}_{z;\eta}:= (\widehat{v}_{z;\eta}(x-z), D \widehat{v}_{z;\eta}(x-z), D^2 \widehat{v}_{z;\eta}(x-z)) \in \Int \, \F_{x-z}, \quad  \forall x \in B. 
			\end{equation}
			Direct calculation yields
			\begin{equation*}\label{NTF5}
			J^2_{x-z} \widehat{v}_{z;\eta} =  (v(x), D_xv(x), D^2_x v(x) + \eta I) =  J^2_x v + (0,0,\eta I).
			\end{equation*} 
			Finally, since $v$ is strictly $\F$-subharmonic and since $\F$ is fiberegular with modulus of continuity $\delta = \delta(\eta)$  (where $|x - (x-z)| = |z| < \delta$), we conclude that
			\begin{equation*}\label{NTF6}
			J^2_{x-z} \widehat{v}_{z;\eta} \in  \Int \F_x + (0,0,\eta I) \subset  \Int \, F_{x-z},  \quad \forall x \in B,
			\end{equation*} 
			which is the needed \eqref{NTF4}.
			
			\noindent \underline{Step 2:} (Subharmonicity of $w_{\eta}^{\veps}$ by way of a Perron procedure)
			
			For each $\eta > 0$ and with $\delta = \delta(\eta)$ the modulus of continuity of $\wt{\F}$, the uniform translation property of Step 1 says that the family of functions defined on $\Omega_{\delta}$
			$$
			\mathfrak{F}_{\eta} := \left\{ w_{\eta;z}(\cdot) := w(\cdot - z) + \frac{\eta}{2} | \cdot |^2: \ \ |z| < \delta  \right\} 
			$$
			consists of $\wt{\F}$-subharmonics on $\Omega_{\delta}$. For each $\veps > 0$, by the negativity property (N) of the subequation $\wt{\F}$, the translated families
			\begin{equation}\label{perron1}
			\mathfrak{F}_{\eta}^{\veps} := \left\{  w_{\eta;z}(\cdot) - \frac{1}{2 \veps} |z|^2 = w(\cdot - z)  - \frac{1}{2 \veps} |z|^2 +  \frac{\eta}{2} |\cdot|^2: \ \ |z| < \delta \right\} 
			\end{equation}
			also consist of $\wt{\F}$-subharmonics on $\Omega_{\delta}$. Since the family $\mathfrak{F}_{\eta}^{\veps}$ is locally uniformly bounded from above, by Proposition B.4 (viii) of \cite{CPR23}, the Perron function defined by
			\begin{equation}\label{perron2}
			v_{\eta}^{\veps}(x) := \sup_{W \in \mathfrak{F}_{\eta}^{\veps}} W(x) = \sup_{\{z: |z| < \delta\}} \left[ w(x - z)  - \frac{1}{2 \veps} |z|^2 +  \frac{\eta}{2} |x|^2 \right], \quad x \in \Omega_{\delta}
			\end{equation}
			satisfies the property that its upper semicontinuous regularization (envelope) 
			\begin{equation}\label{perron3}
			\mbox{ $[v_{\eta}^{\veps}]^{\ast}$ is $\wt{\F}$-subharmonic in $\Omega_{\delta}$,}
			\end{equation}
			where we recall that $[v_{\eta}^{\veps}]^{\ast}$ the smallest function in $\USC(\Omega_{\delta})$ such that $v_{\eta}^{\veps} \leq [v_{\eta}^{\veps}]^{\ast}$ on $\Omega_{\delta}$. 
			
			Finally, using the boundedness of $w$ ($|w| \leq M$ on $\Omega$ and hence on each $\Omega_{\delta}$), we claim that by restricting $\veps$ to be less than
			\begin{equation}\label{perron4}
			\veps_{\ast} = \veps_{\ast}(\delta(\eta), M) := \frac{\delta^2(\eta)}{4M}
			\end{equation}
			one has
			\begin{equation}\label{perron5}
			\mbox{$w_{\eta}^{\veps} = [v_{\eta}^{\veps}]^{\ast}$ \ \ on $\Omega_{\delta}$ \ \ for each $\veps \in (0, \veps_{\ast}]$,}
			\end{equation}
			which by \eqref{perron3} completes the proof if we can justify \eqref{perron5}. To see that \eqref{perron5} holds, first note that by \eqref{perron2} one has
			\begin{equation}\label{perron6}
			v_{\eta}^{\veps}(x)  = \sup_{\{z: |z| < \delta\}} \left[ w(x - z)  - \frac{1}{2 \veps} |z|^2 \right]  +  \frac{\eta}{2} |x|^2 , \quad x \in \Omega_{\delta}.
			\end{equation}
			Next, by resticting to $\veps \in (0, \veps_{\ast}]$ with $\veps_{\ast}$ defined in \eqref{perron4}, as shown in section 8 of \cite{HL09} one has the following identity for the %\note{recall the calculation?} 
			supconvolution \eqref{approx2A}
			\begin{equation}\label{approx3A}
			w^{\veps}(x) = \sup_{\{ y: |y - x| < \delta \}} \left( w(y) - \frac{1}{2 \veps} |y - x|^2 \right), \quad x \in \Omega_{\delta},
			\end{equation}
			and then by making a change of variables, one has the identity
			\begin{equation}\label{approx4A}
			w^{\veps}(x)= \sup_{\{z: |z| < \delta \}} \left( w(x - z) - \frac{1}{2 \veps} |z|^2 \right), \quad \forall \, x \in \Omega_{\delta}.
			\end{equation}
			Combining \eqref{approx4A} with \eqref{perron6} one has 
			$$
			\mbox{$w_{\eta}^{\veps} = v_{\eta}^{\veps}$ \ \ on $\Omega_{\delta}$ \ \ for each $\veps \in (0, \veps_{\ast}]$,}
			$$
			where $w_{\eta}^{\veps}$ is  $\frac{1}{\veps}$-semiconvex on $\Omega \supset \Omega_{\delta}$ and hence upper semicontinuous so that \eqref{perron5} holds.
		\end{proof}
		
	\end{appendix}


\begin{thebibliography}{99} 
		%\bibitem{Br20} K.K.\ Brustad, {\em Ont he comparison principle for second order elliptic equations without first and zeroth order terms}, {\tt arXiv:2008.08399v1}, 19 August 2020, 1--25.
		\bibitem{BP21} I.\ Birindelli and K.\ R.\ Payne, {\em Principal eigenvalues for $k$-Hessian operators by maximum principle methods},  Math.\ Eng.\ {\bf 3} \ (2021), Paper No.\ 021, 37 pp.
		\bibitem{CC} L.\ Caffarelli and X.\ Cabr\'{e}, {\em Fully Nonlinear Elliptic Equations}, American Mathematical Society Colloquiium Publications Vol.\ 43, American Mathematical Society, Providence, RI, 1995.
		\bibitem{CHLP20} M.\ Cirant, F.R.\ Harvey, H.B.\ Lawson, Jr.\ and K.R.\ Payne, {\em Comparison Principles for General Potential Theories and PDEs}, Annals of Mathematics Studies {\bf  Vol.\ 218}, Princeton University Press, Princeton, NJ, 2023, xiv+203 pp. 
		\bibitem{CP17} M.\ Cirant and K.R.\ Payne, {\em On viscosity solutions to the Dirichlet problem for elliptic branches of inhomogeneous fully nonlinear equations}, Publ.\ Mat.\ {\bf 61}\ (2017), 529--575.
		\bibitem{CP21} M.\ Cirant and K.R.\ Payne, {\em Comparison principles for viscosity solutions of elliptic branches of fully nonlinear equations independent of the gradient}, to appear in Math.\ Eng.\ {\tt doi:10.3934/mine.2021030.}.
		\bibitem{CPR23} M.\ Cirant, K.R.\ Payne and D.F.\ Redaelli  {\em Comparison principles for nonlinear potential theories and PDEs with fiberegularity and sufficient monotonicity}, Nonlinear Analysis {\bf 235}\ (2023), Paper No.\ 113343, 52 pp.
		\bibitem{CIL92} M.G.\ Crandall, H.\ Ishii and P-L.\ Lions, {\em User's guide to viscosity solutions of second order partial differential equations}, Bull.\ Amer.\ Math.\ Soc.\ {\bf 27}\ (1992), 1--67.
		\bibitem{Di23} S.\ Dinew, {\em Interior estimates for p-plurisubharmonic functions},
		Indiana Univ.\ Math.\ J.\ {\bf 72}\ (2023), 2025--2057.
		\bibitem{Ga59} L.\ G{\aa}rding, {\em An inequality for hyperbolic polynomials}, J.\ Math. Mech.\ {\bf 8}\ (1959), 957--965.
		\bibitem{GP24a} B.\ Guo and D.H.\ Phong, {\em On $L^{\infty}$ estimates for fully non-linear partial differential equations}, Ann.\ Math.\ {\bf 200}\ (2024), 365--398.
		\bibitem{GP24b} B.\ Guo and D.H.\ Phong, {\em Uniform entropy and energy bounds for fully non-linear equations}, Comm.\ Anal.\ Geom.\ {\bf 32}\ (2024), 2135--2186.
		\bibitem{GPT23} B.\ Guo, D.H.\ Phong and F.\ Tong, {\em On $L^{\infty}$ estimates for complex Monge-Amp\`ere equations}, Ann.\ Math.\ {\bf 198}\ (2023), 393--418.
		%\bibitem{HaLi11} Q.\ Han and F.\ Lin, {\em Elliptic Partial Differential Equations, Second Edition}, Courant Lecture Note Math., {\bf Vol.\ 1}, Courant Institute of Mathematical Sciences, New York; American Mathematical Society, Providence, RI, 2011. x+147 pp.
		\bibitem{HanLin11} Q. Han and F.H.\ Lin, {\em Elliptic Partial Differential Equations, Second Edition}, Courant Lect.\ Notes Math.\ {Vol. 1}, Courant Institute of Mathematical Sciences, New York; American Mathematical Society, Providence, RI, 2011. x+147 pp.
		\bibitem{HL09} F.R.\ Harvey and H.B.\ Lawson, Jr., {\em Dirichlet duality and the nonlinear Dirichlet problem}, Comm.\ Pure Appl.\ Math.\ {\bf 62}\ (2009), 396--443.
		\bibitem{HL10} F.R.\ Harvey and H.B. Lawson, Jr., {\em Hyperbolic polynomials and the Dirichlet problem}, {\tt arXiv:0912.5220v2}, 19 March 2010, 1--48.
		\bibitem{HL11} F.R.\ Harvey and H.B.\ Lawson, Jr., {\em Dirichlet duality and the nonlinear Dirichlet problem on Riemannian manifolds}, J.\ Differential Geom.\ {\bf 88}\ (2011), 395--48.
		\bibitem{HL13a} F.R.\ Harvey and H.B. Lawson, Jr., {\em G\aa rding's theory of hyperbolic polynomials}, Comm.\ Pure Appl.\ Math.\ {\bf 66}\ (2013), 1102--1128.
		%\bibitem{HL13} F.R.\ Harvey and H.B.\ Lawson, Jr., {\em The equivalence of viscosity and distributional subsolutions for convex subequations—a strong Bellman principle}, Bull.\ Braz.\ Math. Soc.\ (N.S.) {\bf 44}\ (2013), 621--652.
		%\bibitem{HL16a} F.R.\ Harvey and H.B. Lawson, Jr., {\em Notes on the differentiation of uasi-convex functions}, {\tt arXiv:1309.1772v3}, 30 July 2016, 1--17.
		%\bibitem{HL16b} F.R.\ Harvey and H.B. Lawson, Jr., {\em The ae theorem and addition theorems for quasi-convex functions}, {\tt arXiv:1309.1770v3}, 30 July 2016, 1--12.
		\bibitem{HL18b} F.R.\ Harvey and H.B. Lawson, Jr., {\em The inhomogeneous Dirichlet Problem for natural operators on manifolds}, Ann.\ Inst.\ Fourier (Grenoble) {\bf 69} \ (2019), 3017--3064.
		\bibitem{HL18c} F.R.\ Harvey and H.B. Lawson, Jr., {\em Tangents to subsolutions: existence and uniqueness, Part I.} Ann.\ Fac.\ Sci.\ Toulouse Math.\ {\bf 27} \ (2018), 777--848
		\bibitem{HL23} F.R.\ Harvey and H.B. Lawson, Jr., {\em Determinant majorization and the work of Guo-Phong-Tong
			and Abja-Olive} Calc.\ Var.\ Partial Differential Equations\ {\bf 62} \ (2023), Paper No. 153, 28 pp.
		\bibitem{HL24} F.R.\ Harvey and H.B. Lawson, Jr., {\em A definitive determinant majorization result for nonlinear operators}, Duke Math.\ J., to appear. Preprint version available online {\tt arxiv.org/abs/2407.05408v1}, 13 pages. 
			\bibitem{HP23} F.R.\ Harvey and K.R.\ Payne, {\em Interplay between nonlinear potential theory and fully nonlinear elliptic PDEs}. Pure Appl.\ Math.\ Q.\ {\bf 19}\ (2023), 2973--3018.
				\bibitem{HP25} F.R.\ Harvey and K.R.\ Payne, {\em The Correspondence Principle: A bridge between general potential theories and nonlinear elliptic differential operators},  {\tt arXiv:2509.13709v1}, 17 Sep 2025, 1--24, submitted for publication.
	%\bibitem{HL19} F.R.\ Harvey and H.B. Lawson, Jr., {\em A generalization of PDEs from a Krylov point of view}, to appear in Adv.\ Math.\ {\tt http://doi.org/10.1016/j.aim.2020.107298}.
		%\bibitem{HL20} F.R.\ Harvey and H.B. Lawson, Jr., {\em Pseudoconvexity for the special Lagrangian potential equation}, to appear in Calc.\ Var.\ Partial Differential Equations, preprint available at \ {\tt arxiv.org/pdf/2001.09818.pdf}.
		%\bibitem{Kv95} N.V.\ Krylov, {\em On the general notion of fully nonlinear second-order elliptic equations}, Trans.\ Amer.\ Math.\ Soc.\ {\bf 347} (1995), 857--895.
		\bibitem{PR25}  K.R.\ Payne and D.F.\ Redaelli - {\em Primer on Semiconvex Functions in General Potential Theories},  Springer Lecture Notes in Mathematics {\bf Vol.\ 2371}, Springer-Verlag, vii + 139 pp.
		\bibitem{R06} J.\ Renegar, {\em Hyperbolic programs, and their derivative relaxations}, Fuound.\ Comput.\ Math.\ {\bf 6}\ (2006), 59--79.
		\bibitem{T25} T. Tao, {\em A variant of Maclaurin's inequality},
		Proc.\ Amer.\ Math.\ Soc.\ Ser.\ B \ {\bf 12}\ (2025), 1–13.
		\bibitem{TW97} N.S.\ Trudinger and X.-J.\ Wang, {\em Hessian measures I},  Topol.\ Methods Nonlinear Anal.\ {\bf 10}\ (1997), 225--239.
		\bibitem{TW99} N.S.\ Trudinger and X.-J.\ Wang, {\em Hessian measures II}, Ann.\ of Math.\ {\bf 150}\ (1999), 579--604.





	
\end{thebibliography}
\end{document}